\documentclass[a4paper,11pt]{article}
\usepackage[utf8]{inputenc}
\usepackage[T1]{fontenc}
\usepackage{amsthm,amsmath,amssymb}
\usepackage[colorlinks=true,citecolor=black,linkcolor=black,urlcolor=blue]{hyperref}
 \usepackage[all]{xy}
\usepackage{geometry}
\theoremstyle{plain}
\newtheorem{theorem}{Theorem}

\newtheorem{corollary}[theorem]{Corollary}
\newtheorem{proposition}[theorem]{Proposition}

\newtheorem{claim}[theorem]{Claim}

\theoremstyle{definition}
\newtheorem{definition}[theorem]{Definition}
\newtheorem{example}[theorem]{Example}

\theoremstyle{remark}
\newtheorem{remark}[theorem]{Remark}

\def\S{{\mathfrak  S}}

\def\<{\langle}
\def\>{\rangle}

\def\C{{\mathbb C}}

\def\N{{\mathbb N}}
\def\Y{{\mathbb Y}}

\def\A{{\sf A}}

\def\ie{{\it i.e.}, }

\def\goth{\mathfrak}

\usepackage{tikz}

\newcommand{\seqnum}[1]{\href{http://oeis.org/#1}{\underline{#1}}}

\def\shuff#1#2{\mathbin{
\hbox{\vbox{ \hbox{\vrule \hskip#2 \vrule height#1 width 0pt
}%
\hrule}%
\vbox{ \hbox{\vrule \hskip#2 \vrule height#1 width 0pt
\vrule }%
\hrule}%
}}}

\def\shuf{{\mathchoice{\shuff{7pt}{3.5pt}}%
{\shuff{6pt}{3pt}}%
{\shuff{4pt}{2pt}}%
{\shuff{3pt}{1.5pt}}}}%
\def\shuffle{\,\shuf\,}

\def\ashuff#1#2#3{
\kern 1pt \vrule height#1 \overline{\vrule height#3 width 0pt
\hskip#2} \rule{.3pt}{#1}\overline{\vrule height#3 width 0pt
\hskip#2} \rule{.3pt}{#1} \kern 1pt }

\def\A{{\mathbb X}}\def\X{{\mathbb X}}\def\Y{{\mathbb Y}}
\def\B{{\mathbb Y}}

\def\A{{\mathbb A}}
\def\B{{\mathbb B}}

\def\Carre3#1{\left[\begin{array}{ccc}#1\end{array}\right]}

\def\S{{\goth S}}
\def\WSym{\mathbf{WSym}}

\def\PiQSym{\mathrm{\Pi QSym}}
  
\def\BPiQSym{\mathrm{B\Pi QSym}}
\def\CPiQSym{\mathrm{C\Pi QSym}}

\def\std{\mathrm{std}}

\def\std{\mathrm{std}}

\begin{document}
\title{Word Bell Polynomials}
\author{Ammar Aboud\footnote{aboudam@gmail.com. 
USTHB, Faculty of Mathematics, Po. Box 32 El Alia 16111 Algiers, Algeria.}, Jean-Paul Bultel\footnote{jean-paul.bultel@univ-rouen.fr. Laboratoire LITIS - EA 4108, Université de Rouen, Avenue de l’Université - BP 8
76801 Saint-Étienne-du-Rouvray Cedex }, Ali Chouria\footnote{ali.chouria1@univ-rouen.fr. Laboratoire LITIS - EA 4108, Université de Rouen, Avenue de l’Université - BP 8
76801 Saint-Étienne-du-Rouvray Cedex }\\  Jean-Gabriel Luque\footnote{jean-gabriel.luque@univ-rouen.fr. Laboratoire LITIS - EA 4108, Université de Rouen, Avenue de l’Université - BP 8
76801 Saint-Étienne-du-Rouvray Cedex } and Olivier Mallet\footnote{olivier.mallet@univ-rouen.fr. Laboratoire LITIS - EA 4108, Université de Rouen, Avenue de l’Université - BP 8
76801 Saint-Étienne-du-Rouvray Cedex }}
\maketitle
\noindent{\footnotesize {\bf Keywords:} {Bell polynomials, Symmetric functions, Hopf algebras, Fa\`a di Bruno algebra, Lagrange inversion, Word symmetric functions, Set partitions.}
\abstract{Partial multivariate Bell polynomials have been defined by  E.T. Bell in 1934. These polynomials have numerous applications in Combinatorics, Analysis, Algebra, Probabilities \emph{etc.}
Many of the formul\ae\ on Bell polynomials involve combinatorial objects (set partitions, set partitions into lists, permutations \emph{etc}).  So it seems natural to investigate analogous formul\ae \   in some combinatorial Hopf algebras with bases indexed by these objects. 
In this paper we investigate the connexions between Bell polynomials and several combinatorial Hopf algebras: the Hopf algebra of symmetric functions, the Fa\`a di Bruno algebra, the Hopf algebra of word symmetric functions etc. We show that Bell polynomials can be defined in all these algebras and we give analogues of classical results. To this aim, we construct and study a family of combinatorial Hopf algebras whose bases are indexed by colored set partitions.} 



\section{Introduction}
Partial multivariate Bell polynomials (Bell polynomials for short) have been defined by E.T. Bell in \cite{bell1934exponential} in 1934. But their name is due to Riordan \cite{riordan1958introduction} which studied the Fa\`a di Bruno formula \cite{faadibruno1855sviluppo,faadibruno1857note} allowing one to write the $n$th derivative of a composition $f\circ g$ in terms of the derivatives of $f$ and $g$ \cite{riordan1946derivative}. The applications of Bell polynomials in Combinatorics, Analysis, Algebra, Probabilities \emph{etc.} are so numerous that it should be very long to detail them in the paper. Let us give only a few seminal examples.
\begin{itemize}
\item The main applications to Probabilities follow from the fact that the $n$th moment of a probability distribution is a complete Bell polynomial of the cumulants.
\item Partial Bell polynomials are linked to the Lagrange inversion. This follows from the Fa\`a di Bruno formula.
\item Many combinatorial formul\ae\ on Bell polynomials involve classical combinatorial numbers like Stirling numbers, Lah numbers \emph{etc.}
\end{itemize}

The Fa\`a di Bruno formula and many combinatorial identities can be found in \cite{comtet1974advanced}.
The PhD thesis of M. Mihoubi \cite{ThesisMih} contains a rather complete survey of the applications of these polynomials together with numerous formul\ae.

Some of the simplest formul\ae\  are related to the enumeration of combinatorial objects (set partitions, set partitions into lists, permutations \emph{etc.}). So it seems natural to investigate analogous formul\ae \   in some combinatorial Hopf algebras with bases indexed by these objects.
We recall that combinatorial Hopf algebras are graded bigebras with bases indexed by combinatorial objects such that the product
 and the coproduct have some compatibilities. 

The paper is organized as follows. In Section \ref{SetPart}, we investigate the combinatorial properties of the colored set partitions. 
Section \ref{Hopf} is devoted to the study of the Hopf algebras of colored set partitions. 
After having introduced this family of algebras, we give some special cases which can be found in the literature. 
The main application explains the connections with $Sym$, the algebra of symmetric functions. This explains that we can recover some identities on Bell polynomials when the variables are specialized to  combinatorial numbers from analogous identities in some combinatorial Hopf algebras. We show that the algebra ${\bf WSym}$ of word symmetric functions has an important role for this construction.
In Section \ref{SBell}, we give a few analogues of complete and partial Bell polynomials in ${\bf WSym}$, $\PiQSym={\bf WSym}^*$ and $\C\langle \A\rangle$ and investigate their main properties.
Finally, in Section \ref{Munthe-Kaas} we investigate the connection with 
other noncommutative analogues of Bell polynomials defined by Munthe-Kass \cite{MK}.

\section{Definition, background and basic properties of colored set partitions}\label{SetPart}

\subsection{Colored set partitions}

Let $a=(a_{m})_{m \geq 1}$ be a sequence of nonnegative integers.
A \emph{colored set partition} associated to the sequence $a$ is a set of couples
$$\Pi=\{ [\pi_1, i_1],  [\pi_2, i_2], \dots ,  [\pi_k, i_k] \}$$ such that
$\pi=\{\pi_1,\dots,\pi_k\}$ is a partition of $\{1,\dots,n\}$ for some $n\in\mathbb N$ and $1\leq i_\ell\leq  a_{\#\pi_{\ell}}$ for each $1\leq\ell\leq k$.  The integer $n$ is the \emph{size} of $\Pi$. We will write $|\Pi|=n$, $\Pi\vDash n$ and $\Pi\Rrightarrow \pi$.
We will denote by $\mathcal{CP}_n(a)$ the set of colored partitions of size $n$ associated to the sequence $a$. Notice that these sets are finite. We will set also $\mathcal{CP}(a)=\bigcup_n\mathcal{CP}_n(a)$. We  endow $\mathcal{CP}$ with the additional statistic $\#\Pi$ and set $\mathcal{CP}_{n,k}(a)=\{\Pi\in\mathcal{CP}_{n}(a):\#\Pi=k\}.$
\begin{example}\rm Consider the sequence whose first terms are $a=(1,2,3,\dots)$. The colored partitions of size $3$ associated to $a$ are
\[\begin{array}{l}
\mathcal{CP}_3(a)=\{\{[\{1,2,3\},1]\}, 
\{[\{1,2,3\},2]\}, \{[\{1,2,3\},3]\}, 
\{[\{1,2\},1],[\{3\},1]\},\\
\{[\{1,2\},2],[\{3\},1]\},
\{[\{1,3\},1],[\{2\},1]\},
\{[\{1,3\},2],[\{2\},1]\},\\
\{[\{2,3\},1],[\{1\},1]\},
\{[\{2,3\},2],[\{1\},1]\},
\{[\{1\},1],[\{2\},1],[\{3\},1]\}
  \}.
\end{array}
\]
The colored partitions of size $3$ and cardinality $2$ are
\[\begin{array}{rcl}
\mathcal{CP}_{3,2}(a)&=&\{ 
\{[\{1,2\},1],[\{3\},1]\},
\{[\{1,2\},2],[\{3\},1]\},
\{[\{1,3\},1],[\{2\},1]\},\\&&
\{[\{1,3\},2],[\{2\},1]\},
\{[\{2,3\},1],[\{1\},1]\},
\{[\{2,3\},2],[\{1\},1]\}
  \}.
\end{array}
\]
\end{example}

It is well known that the number of colored set partitions of size $n$ for a given sequence $a=(a_n)_n$ is equal to the evaluation of the complete Bell polynomial $A_n(a_1,\dots,a_m,\dots)$ and that the number of colored set partitions of size $n$ and cardinality $k$ is given by the evaluation of the partial Bell polynomial $B_n(a_1,a_2,\dots,a_m,\dots)$. That is
\[
\#\mathcal{CP}_n(a)=A_n(a_1,a_2,\dots)\mbox{ and }\#\mathcal{CP}_{n,k}(a)=B_{n,k}(a_1,a_2,\dots).
\]

Now, let $\Pi=\{[\pi_1,i_1],\ldots [\pi_k,i_k]\}$ be a set such that the $\pi_j$ are finite sets of nonnegative integers such that no integer belongs to more than one $\pi_j$, and $1\leq i_j\leq a_{\#(\pi_j)}$ for each $j$. Then, the \emph{standardized} $\std(\Pi)$ of $\Pi$ is well defined as the unique colored set partition obtained by replacing the $i$th smallest integer in the $\pi_j$ by $i$. 
\begin{example}\rm For instance:
\[\begin{array}{l}\std(\{[\{1,4,7\}, 1],[\{3,8\}, 1],[\{5\}, 3],[\{10\}, 1]\}) =\\ \{[\{1,3,5\}, 1],[\{2,6\}, 1],[\{4\}, 3],[\{7\}, 1]\}\end{array} \]
\end{example}
We define two binary operations $\uplus:\mathcal{CP}_{n,k}(  a)\otimes\mathcal{CP}_{n',k'}(  a)\longrightarrow \mathcal{CP}_{n+n',k+k'}(  a)$,
\[\Pi \uplus \Pi' = \Pi\cup \Pi'[n],\] 
where $\Pi'[n]$ means that we add $n$ to each integer occurring in the sets of $\Pi'$ 
 and \\$\Cup$: $\mathcal{CP}_{n,k}\otimes\mathcal{CP}_{n',k'}\longrightarrow \mathcal P(\mathcal{CP}_{n+n',k+k'})$ by
\[
\Pi\Cup\Pi'=\{\hat\Pi \cup \hat\Pi' \in\mathcal{CP}_{n+n',k+k'}(  a):\std(\hat\Pi)=\Pi\mbox{ and }
\std(\hat\Pi')=\Pi'\}.
\]
\begin{example}
\rm We have
\[\{[\{1,3\},5],[\{2\},3]\}\uplus \{[\{1\},2],[\{2,3\},4]\}=\{[\{1,3\},5],[\{2\},3],[\{4\},2],[\{5,6\},4]\},\]
and
\[\begin{array}{l}
\{[\{1\},5],[\{2\},3]\}\Cup \{[\{1,2\},2]\} =\{\{[\{1\},5],[\{2\},3],[\{3,4\},2]\},\\ \{[\{1\},5],[\{3\},3],[\{2,4\},2]\},
 \{[\{1\},5],[\{4\},3],[\{2,3\},2]\},\\ \{[\{2\},5],[\{3\},3],[\{1,4\},2]\},
\{[\{2\},5],[\{4\},3],[\{1,3\},2]\},\{[\{3\},5],[\{4\},3],[\{1,2\},2]\}\}.
\end{array}\]
\end{example}
The operator $\Cup$ provides an algorithm which computes all the colored partitions:
\begin{equation}\label{allcpart}
\mathcal{CP}_{n,k}(  a)=\bigcup_{i_1+\dots+i_k=n}\bigcup_{j_1=1}^{a_{i_1}}\cdots
\bigcup_{j_k=1}^{a_{i_k}}\{[\{1,\dots,i_1\}, j_1]\}\Cup\cdots\Cup\{[\{1,\dots,i_k\}, j_k]\}.
\end{equation}

Nevertheless each colored partition is generated more than once using this process.
For a triple $(\Pi,\Pi',\Pi'')$ we will denote by $\alpha_{\Pi',\Pi''}^{\Pi}$ the number of pairs of disjoint subsets $(\hat \Pi'$, $\hat\Pi'')$ of $\Pi$ such that $\hat\Pi'\cup\hat\Pi''=\Pi$, $\std(\hat\Pi')=\Pi'$ and $\std(\hat\Pi'')=\Pi''$.
\begin{remark}\label{recpart}
Notice that for $  a=\mathbf 1=(1,1,\dots)$ (\emph{i.e.} the ordinary set partitions), there is an alternative simple way to construct efficiently the set $\mathcal{CP}_n(\mathbf 1)$. It suffices to use the induction
\begin{equation}\label{allpart}
\begin{array}{ll}
\mathcal{CP}_0(\mathbf 1)=&\{\emptyset\},\\ 
\mathcal{CP}_{n+1}(\mathbf 1)=&\{\pi\cup\{\{n+1\}\}:\pi\in\mathcal{CP}_n (\mathbf 1)\}
\cup\{(\pi\setminus\{e\})\\&\cup\{e\cup\{n+1\}\}:\pi\in \mathcal{CP}_n (\mathbf 1),\ e\in\pi\}\}.
\end{array}
\end{equation}
Applying this recurrence, the set partitions of $\mathcal{CP}_{n+1}(\mathbf 1)$ are each obtained exactly once from the set partitions of $\mathcal{CP}_{n}(\mathbf 1)$.
\end{remark}
\subsection{Generating series}
The generating series of the colored set partitions $\mathcal{CP}(a)$ is  obtained from the cycle generating function 
for the species of colored set partitions. 
The construction is rather classical, see e.g. \cite{BLL}. Recall first that a species of structures is a rule $F$ which produces for each 
finite set $U$, a finite set $F[U]$ and for each bijection $\phi:U\longrightarrow V$, a function $F[\phi]:F[U]\longrightarrow F[V]$ satisfying the following properties:
\begin{itemize}
\item for all pairs of bijections $\phi:U\longrightarrow V$ and $\psi:V\longrightarrow W$, $F[\psi\circ\phi]=F[\psi]\circ F[\phi]$
\item if $Id_U$ denotes the identity map on $U$, then $F[Id_U]=Id_{F[U]}$.
\end{itemize}
An element $s\in F[U]$ is called an $F$-structure on $U$.
The cycle generating function of a species $F$ is the formal power series in infinitely independent
many variables $p_1, p_2,\dots$ (called power sums) defined by the formula
\begin{equation}
Z_F(p_1,p_2,\cdots)=\sum_{n=0}^{\infty}\frac{1}{n!}\sum_{\sigma\in\S_n}|F([n])^\sigma|p^{\mbox{cycle\_type}(\sigma)},
\end{equation}
where $F([n])^\sigma$ denotes the set of $F$-structures on $[n]:=\{1,\dots,n\}$ which are fixed by the permutation $\sigma$, and   
$p^{\lambda}=p_{\lambda_1}\cdots p_{\lambda_k}$ if $\lambda$ is the vector $[\lambda_1,\dots,\lambda_k]$.
For instance, the trivial species $\mathrm{TRIV}$ has  only one $\mathrm{TRIV}$-structure on every $n$. Hence, its cycle generating function
is nothing else but the Cauchy function
\begin{equation}
\sigma_1:=\exp\left\{\sum_{n\geq 1}\frac{p_n}{n}t^n\right\}=\sum_{n\geq 0}h_n.
\end{equation}
Here $h_n$ denotes the complete function $h_n=\sum_{\lambda\vdash n}\frac1{z_\lambda}p^\lambda$ where $\lambda\vdash n$ 
means that the sum is over the partitions $\lambda$
of $n$ and $z_\lambda=\prod_{i}i^{m_i(\lambda(i))}m_i(\lambda(i))!$ if $m_i(\lambda(i))$ is the multiplicity of 
the part $i$ in $\lambda$.\\
We consider also the species $\mathrm{NCS}(a)$ of non-empty colored sets having $a_n$ $\mathrm{NCS}(a)$-structures 
on $[n]$ which are invariant by permutations.
Its cycle generating function is
\begin{equation}
Z_{\mathrm{NCS}(a)}=\sum_{n\geq 1}a_nh_n.
\end{equation}

As a species, $\mathcal{CP}(a)$ is a composite $\mathrm{TRIV}\circ \mathrm{NCS}(a)$.
 Hence, its cycle generating function
is obtained by computing the plethysm
\begin{equation}\label{cycleind}
Z_{\mathrm{NCS}(a)}(p_1,p_2,\dots)=\sigma_1[Z_{\mathrm{NCS}(a)}]=\exp\left\{\sum_{n>0}\frac{1}n\sum_{k>0}a_kp_n[h_k]\right\}.
\end{equation}
The exponential generating function of $\mathcal{CP}(a)$ is obtained by setting $p_1=t$ and $p_i=0$ for $i>1$ in (\ref{cycleind}):
\begin{equation}\label{Bell}
\sum_{n\geq 0}A_n(a_1,a_2,\dots)\frac{t^n}{n!}=\exp\left\{\sum_{i>0}\frac{a_i}{i!}t^i\right\}.
\end{equation} 
We deduce easlily that $A_n(a_1,a_2,\dots)$ are multivariate polynomials in the variables $a_i$'s. 
These polynomials are known under the name of
complete Bell polynomials \cite{bell1934exponential}.
 The double generating function of $\mathrm{card}(\mathcal{CP}_{n,k}(  a))$ is easily deduced from (\ref{Bell}) by
\begin{equation}\label{Bell2}
\sum_{n\geq 0}\sum_{k\geq 0}B_{n,k}(a_1,a_2,\dots){x^kt^n\over n!}=\exp\left\{x\sum_{i>0}{a_i\over i!}t^i\right\}.
\end{equation} 
Hence, 
\begin{equation}\label{partial Bell}
\sum_{n\geq k}B_{n,k}(a_1,a_2,\dots){t^n\over n!}=\frac1{k!}\left(\sum_{i>0}{a_i\over i!}t^i\right)^n.
\end{equation} 
So, one has
\begin{equation}
A_{n}(a_1, a_2, \dots) = \sum_{k=1}^{n} B_{n,k}(a_1, a_2, \dots), \forall n \geqslant 1
\text{ and }A_{0}(a_1, a_2, \dots) = 1.
\end{equation}

The multivariate polynomials $B_{n,k}(a_1,a_2,\dots)$ are known under the name partial Bell polynomials \cite{bell1934exponential}.
Let $S_{n,k}$ denote the Stirling number of the second kind which counts the number of ways to partition a set of $n$ objects into $k$ nonempty subsets. The following identity holds
\begin{equation}\label{stirling2}B_{n,k}(1,1, \dots)= S_{n,k}. \end{equation}
Note also that $A_{n}(x,x,\dots) = \sum_{k=0}^{n} S_{n,k} x^{k}$ is the classical univariate Bell polynomial denoted by $\phi_n(x)$ in \cite{bell1934exponential}. Several other identities involve combinatorial numbers. For instance, one has
\begin{equation}\label{lah}B_{n,k}(1!,2!,3!, \dots)= \binom{n-1}{k-1} \frac{n!}{k!}, \text{ Unsigned Lah numbers \seqnum{A105278} in \cite{Sloane}}, \end{equation}
\begin{equation}\label{idempotent}B_{n,k}(1,2,3, \dots)= \binom {n}{k} k^{n-k}, \text{     Idempotent numbers  \seqnum{A059297} in \cite{Sloane}}, \end{equation}
\begin{equation}\label{stirling1}B_{n,k}(0!,1!,2!, \dots)= |s_{n,k}|, \text{ Stirling numbers of the first kind \seqnum{A048994} in \cite{Sloane}}.  \end{equation}

We can also find many other examples in \cite{bell1934exponential, comtet1974advanced, mihoubi2008bell, wang2009general, mihoubi2010partial}.

\begin{remark}\rm Without loss of generality, when needed, we will suppose $a_1=1$ in the remainder of the paper. Indeed, if $a_1\neq 0$, then the generating function gives
\begin{equation}\label{Cas1}
B_{n,k}(a_1,\dots,a_p,\dots)=a_1^kB_{n,k}\left(1,{a_2\over a_1},\cdots,{a_p\over a_1}\right)
\end{equation}
and when $a_1=0$,
\begin{equation}\label{Cas2}
B_{n,k}(0,a_2,\dots,a_p,\dots)=\left\{\begin{array}{ll}0\mbox{ if }n<k\\{n!\over (n-k)!}B_{n,k}(a_2,\dots,a_p,\dots)&\mbox{ if } n\geq k.\end{array}\right. 
\end{equation}
\end{remark}
Notice that  the ordinary series of the isomorphism types of $\mathcal{CP}(a)$ is obtained by setting $p_i=t^i$ in (\ref{cycleind}). 
Remarking that under this specialization we have $p_k[h_n]=t^{nk}$, we obtain, unsurprisingly, the ordinary generating series of colored (integer) partitions
\begin{equation}
\prod_{i>0}{1\over (1-t^i)^{a_i}}.
\end{equation}

\subsection{Bell polynomials and symmetric functions}
The algebra of symmetric functions \cite{macdonald1998symmetric, Lascoux} is isomorphic to its polynomial realization $Sym(\X)$ on an infinite set $\X=\{x_1,x_2, \dots\}$ of commuting variables, so the algebra $Sym(\mathbb{X})$ is defined as the set of polynomials invariant under permutation of the variables.
As an algebra, $Sym(\mathbb{X})$ is freely generated by the power sum symmetric functions $p_n(\X)$, defined by $p_n(\X) = \sum_{i \geqslant 1} x_i^n$, or the complete symmetric functions $h_n$, where $h_n$ is the sum of all the monomials of total degree $n$ in the variables $x_1, x_2, \dots$. 
The generating function for the $h_n$, called \emph{Cauchy function}, is 
\begin{align}
\begin{split}
\sigma_t(\X) = \sum_{n \geqslant 0} h_{n}(\X) t^n
= \prod_{i \geqslant 1} (1 - x_i t)^{-1}.
\end{split}
\label{cauchy}
\end{align}
The relationship between the two families $(p_n)_{n\in\N}$ and $(h_n)_{n\in\N}$ is described in terms of generating series  by the \emph{Newton formula}:  
\begin{equation}\label{newton}
\sigma_{t}(\mathbb{X})= \exp\{ \sum_{n \geqslant 1} p_n(\mathbb{X})\frac{t^{n}}{n} \}.
\end{equation}
Notice that $Sym$ is the free commutative algebra generated by $p_1, p_2 \dots$ \emph{i.e.}
 $Sym = \mathbb{C}[p_1, p_2, \dots]$  and $Sym(\X) = \mathbb{C}[p_1(\X), p_2(\X), \dots]$ 
 when $\X$ is an infinite alphabet without relations on the variables. 
 As a consequence of the Newton Formula (\ref{newton}), 
 it is also the free commutative algebra generated by $h_1, h_2, \dots$. 
 The freeness of the algebra provides a mechanism of specialization.
  For any sequence of commutative scalars $u = (u_n)_{n\in\N}$, 
  there is a morphism of algebra $\phi_{u}$ sending each $p_n$ to $u_n$ 
  (resp. sending $h_n$ to a certain $v_n$ which can be deduced from $u$).
   These morphisms are manipulated \emph{as if} there exists an underlying alphabet 
   (so called \emph{virtual alphabet}) $\X_{u}$ such that $p_n(\X_{u})=u_n$ (resp. $h_n(\X_{u})= v_n$).
    The interest of such a vision is that one  defines  operations on sequences and symmetric functions by manipulating alphabets.

The bases of $Sym$ are indexed by the \emph{partitions} $\lambda\vdash n$ of all the integers $n$.
 A partition $\lambda$ of $n$ is a finite noncreasing sequence of positive integers 
 $(\lambda_1\geq\lambda_2\geq\ldots)$ such that $\sum_i \lambda_i = n$.  
 
 By specializing either the power sums $p_i$ or the complete functions $h_i$ to the numbers $a_i\over i!$, the partial and complete Bell 
 polynomials are identified with well known bases.
 
 The  algebra $Sym$ is usually endowed with three coproducts:
 \begin{itemize}
 \item the coproduct $\Delta$ such that the power sums are Lie-like ($\Delta(p_n)=p_n\otimes 1+1\otimes p_n$);
 \item the coproduct $\Delta'$ such that the power sums are group-like ($\Delta'(p_n)=p_n\otimes p_n$);
 \item the coproduct of Fa\`a di Bruno (see e.g. \cite{doubilet1974hopf, joni1979coalgebras}).
 \end{itemize}

Most of the formulas on Bell polynomials can be stated and proved using specializations and these three coproducts. 
Since this is not really the purpose of our article, we have deferred   to Appendix \ref{Symmetric} a list of examples which are rewrites of existing proofs in terms of symmetric functions.
One of the aims of our paper is to rise some of these identities to other combinatorial Hopf algebras.

\section{Hopf algebras of colored set partitions}\label{Hopf}

\subsection{The Hopf algebras ${\bf CWSym}(  a)$ and $\CPiQSym(  a)$}

Let ${\bf CWSym}(  a)$  (${\bf CWSym}$ for short when there is no ambiguity) be the algebra defined by its basis $(\Phi_\Pi)_{\Pi\in\mathcal {CP}(a)}$ indexed by colored set partitions associated to the sequence $  a=(a_m)_{m\geq 1}$  and the product
\begin{equation}
\Phi_\Pi \Phi_{\Pi'} = \Phi_{\Pi\uplus\Pi'}.
\end{equation}
\begin{example}
\rm 
One has
\[
\Phi_{\{[\{1,3,5\},3],[\{2,4\},1]\}}\Phi_{\{[\{1,2,5\},4],[\{3\},1],[\{4\},2]\}}
=\Phi_{\{[\{1,3,5\},3],[\{2,4\},1],[\{6,7,10\},4],[\{8\},1],[\{9\},2]\}}.
\]
\end{example}
Let ${\bf CWSym}_n$ be the subspace generated by the elements $\Phi_\Pi$ such that $\Pi\vDash n$.

For each $n$ we consider an infinite alphabet $\mathbb{A}_{n}$  of noncommuting variables and we suppose $\mathbb{A}_{n}\cap \mathbb{A}_{m} = \emptyset $ when $n\neq m$. For each  colored set partition $\Pi=\{ [\pi_1, i_1],  [\pi_2, i_2], \dots ,  [\pi_k, i_k] \}$,
we construct a polynomial
$\Phi_\Pi(\A_1,\A_2,\dots)\in\C\left\langle \bigcup_n \mathbb{A}_n\right\rangle$ 
\begin{equation}
\Phi_{\Pi} (\A_1,\A_2,\dots):= \sum_{\mathtt w =\mathtt  a_1 \dots\mathtt  a_n}\mathtt  w ,
\end{equation}
where the sum is over the words $\mathtt w =\mathtt  a_1 \dots\mathtt  a_n$ satisfying
\begin{itemize}
\item For each $1\leq\ell\leq k$, $\mathtt a_j\in \A_{i_\ell}$ if and only if $j\in\pi_\ell$.
\item If $j_1,j_2\in\pi_\ell$ then $\mathtt a_{j_1}=\mathtt a_{j_2}$.
\end{itemize}
\begin{example}\rm
\[
\Phi_{\{[\{1,3\},3],[\{2\},1],[\{4\},3]\}}(\A_1,\A_2,\dots) = \sum_{\substack{\mathtt a_1,\mathtt a_2 \in \mathbb{A}_3 \\\mathtt  b \in \mathbb{A}_1}} \mathtt a_1\mathtt b\mathtt a_1\mathtt a_2 .
\]
\end{example}
\begin{proposition}
The family 
\[\Phi(  a):=(\Phi_\Pi(\A_1,\A_2,\dots))_{\Pi\in\mathcal{CP}(  a)}
\]
spans a subalgebra of $\C\left\langle \bigcup_n \mathbb A_n\right\rangle$ which is isomorphic to ${\bf CWSym}(  a)$.  
\end{proposition}
\begin{proof}
 First, remark that $\mathrm{span}(\Phi(  a))$ is stable under concatenation. Indeed, 
$$\Phi_\Pi(\A_1,\A_2,\dots)\Phi_{\Pi'}(\A_1,\A_2,\dots)=\Phi_{\Pi\uplus \Pi'}(\A_1,\A_2,\dots).$$ Furthermore, this shows that $\mathrm{span}(\Phi(  a))$ is homomorphic to ${\bf CWSym}(  a)$ and that an explicit (onto) morphism is given by $\Phi_\Pi\longrightarrow  \Phi_\Pi(\A_1,\A_2,\dots)$. Observing that the family $\Phi(  a)$ is linearly independent, the fact that the algebra  ${\bf CWSym}(  a)$ is graded in finite dimension implies the result.
\end{proof}
We turn ${\bf CWSym}$ into a Hopf algebra by considering the coproduct
\begin{equation}
\Delta(\Phi_\Pi) = \sum\limits_{\substack{\hat\Pi_1\cup\hat\Pi_2=\Pi\\ \hat\Pi_1\cap\hat\Pi_2=\emptyset}}\Phi_{\std(\hat\Pi_1)}\otimes\Phi_{\std(\hat\Pi_2)}=\sum_{\Pi_1,\Pi_2}\alpha_{\Pi_1,\Pi_2}^{\Pi}\Phi_{\Pi_1}\otimes\Phi_{\Pi_2}.
\end{equation}

Indeed,  ${\bf CWSym}$ splits as a direct sum of finite dimension spaces $${\bf CWSym}=\bigoplus_n{\bf CWSym}_n.$$
This defines a natural graduation on $\bf CWSym$.
 Hence, since it is a connected algebra, it suffices to verify that it is a bigebra. More precisely:
\[\begin{array}{rcl}
\Delta(\Phi_\Pi \Phi_{\Pi'})&=&\Delta(\Phi_{\Pi\uplus\Pi'})\\&=&\displaystyle\sum_{\hat\Pi_1\cup\hat\Pi_2=\Pi,
\hat\Pi'_1\cup\hat\Pi'_2=\Pi'[n]\atop \hat\Pi_1\cap\hat\Pi_2=\emptyset,\hat\Pi'_1\cap\hat\Pi'_2=\emptyset}
\Phi_{\std(\hat\Pi_1)\uplus \std(\hat\Pi'_1)}\otimes\Phi_{\std(\hat\Pi_2)\uplus \std(\hat\Pi'_2)}\\
&=&\Delta(\Phi_\Pi)\Delta( \Phi_{\Pi'}).
\end{array}\]

Notice that $\Delta$ is cocommutative.
\begin{example}
\rm For instance,
\[\begin{array}{rcl}
\Delta\left(\Phi_{\{[\{1,3\},5],[\{2\},3]\}}\right)
&=&\Phi_{\{[\{1,3\},5],[\{2\},3]\}}\otimes 1+
\Phi_{\{[\{1,2\},5]\}}\otimes \Phi_{\{[\{1\},3]\}}
+\\&& \Phi_{\{[\{1\},3]\}}\otimes \Phi_{\{[\{1,2\},5]\}}+
 1\otimes \Phi_{\{[\{1,3\},5],[\{2\},3]\}}.\end{array}
\]
\end{example}

 The graded dual $\CPiQSym( a)$ (which will be called $\CPiQSym$ for short when there is no ambiguity) of ${\bf CWSym}$ is the Hopf algebra generated as a space by the dual basis $(\Psi_\Pi)_{\Pi\in\mathcal {CP}(  a)}$ of $(\Phi_\Pi)_{\Pi\in\mathcal {CP}(  a)}$. Its product and its coproduct are given by
\[
\Psi_{\Pi'}\Psi_{\Pi''}=\sum_{\Pi\in\Pi'\Cup\Pi''}\alpha_{\Pi',\Pi''}^{\Pi} \Psi_{\Pi}\mbox{ and }
\Delta(\Psi_{\Pi})=\sum_{\Pi'\uplus\Pi''=\Pi}\Psi_{\Pi'}\otimes\Psi_{\Pi''}.
\]
\begin{example} \rm
For instance, one has
\begin{align*}
\Psi_{\{[\{1,2\},3]\}} \Psi_{\{[\{1\},4],[\{2\},1]\}} & = \Psi_{\{[\{1,2\}, 3], [\{3\}, 4], [\{4\}, 1]\}} 
+ \Psi_{\{[\{1,3\}, 3], [\{2\}, 4], [\{4\}, 1]\}} \\&+ \Psi_{\{[\{1,4\}, 3], [\{2\}, 4], [\{3\}, 1]\}}  + \Psi_{\{[\{2,3\}, 3], [\{1\}, 4], [\{4\}, 1]\}}\\& + \Psi_{\{[\{2,4\}, 3], [\{1\}, 4], [\{3\}, 1]\}} + \Psi_{\{[\{3,4\}, 3], [\{1\}, 4], [\{2\}, 1]\}}
\end{align*}
and
\begin{align*}
\Delta{(\Psi_{\{[\{1,3\},3],[\{2\},4],[\{4\},1]\}})}
& = 1 \otimes \Psi_{\{[\{1,3\},3],[\{2\},4],[\{4\},1]\}} + \Psi_{\{[\{1,3\},3],[\{2\},4]\}} \otimes \Psi_{\{[\{1\},1]\}} \\& + \Psi_{\{[\{1,3\},3],[\{2\},4],[\{4\},1]\}} \otimes 1 .
\end{align*}
\end{example}

\subsection{Special cases}
In this section, we investigate a few interesting special cases of the construction.
\subsubsection{Word symmetric functions}
The most prominent example follows from the specialization $a_n=1$ for each $n$. In this case, the Hopf algebra ${\bf CWSym}$ is isomorphic to ${\bf WSym}$, the Hopf algebra of word symmetric functions. Let us briefly recall its construction. The algebra of word symmetric functions is a way to construct a noncommutative analogue of the
algebra $Sym$. Its bases are indexed by set partitions. After the seminal paper \cite{wolf1936symmetric}, this algebra was investigated
 in \cite{bergeron2008invariants,hivert2008commutative}  as well as  an abstract algebra as in its realization with noncommutative variables. Its name
comes from its realization as a subalgebra of $\C\langle \A\rangle$ where $\A = \{\mathtt a_1,\dots,  \mathtt a_n\cdots \}$  is an infinite alphabet. 

Consider the family of functions $\Phi:=\{\Phi_\pi\}_\pi$ whose elements are indexed by set partitions of $\{1,\ldots,n\}$.
The algebra ${\bf WSym}$  is formally generated by $\Phi$ using the shifted concatenation product:
$
\Phi_\pi\Phi_{\pi'} = \Phi_{\pi\pi'[n]}
$
where $\pi$ and $\pi'$ are set partitions of $\{1,\ldots,n\}$ and $\{1,\ldots,m\}$, respectively, and $\pi'[n]$ means that we add $n$ to each integer occurring in $\pi'$. The polynomial realization ${\bf WSym}(\A)\subset\C\langle \A\rangle$ is defined by $\Phi_\pi(\A)=\sum_{\mathtt w}\mathtt w$ where the sum is over the words $\mathtt w=\mathtt a_1\cdots \mathtt a_n$ where $i,j\in \pi_\ell$ implies $\mathtt a_i=\mathtt a_j$, if $\pi=\{\pi_1,\dots,\pi_k\}$ is a set partition of $\{1,\dots,n\}$.
\begin{example} \rm For instance, one has
 $
 \Phi_{\{\{1,4\}\{2,5,6\}\{3,7\}\}}(\A) = \sum_{\mathtt a, \mathtt b, \mathtt c \in \A} \mathtt  {abcabbc}.
 $
 \end{example}
Although the construction of ${\bf WSym}(\A)$, the polynomial realization of $\WSym$,  seems to be close to $Sym(\X)$, the structures of the two algebras are quite different since  the Hopf algebra ${\bf WSym}$ is not autodual. Surprisingly, the graded dual $\PiQSym:={\bf WSym}^*$ of ${\bf WSym}$ admits a realization in the same  subspace (${\bf WSym}(\A)$) of $\C\langle \A\rangle$ but for the shuffle product. 

With no surprise, we notice the following fact: 
\begin{proposition}~ 
\begin{itemize}
\item
The algebras ${\bf CWSym}(1,1,\dots)$, ${\bf WSym}$ and ${\bf WSym}(\A)$ are isomorphic.\item 
The algebras $\CPiQSym(1,1,\dots)$, $\PiQSym$ and $({\bf WSym}(\A),\shuffle)$ are isomorphic.
\end{itemize}
\end{proposition}
In the rest of the paper, when there is no ambiguity, we will identify the algebras ${\bf WSym}$ and ${\bf WSym}(\A)$.

The word analogue of the basis $(c_\lambda)_\lambda$ of $Sym$ is the dual basis $(\Psi_\pi)_\pi$ of $(\Phi_\pi)_\pi$.

Other bases are known, for example, the word monomial functions defined by
$
\Phi_\pi = \sum_{\pi \le \pi'} M_{\pi'}
$, where $\pi \le \pi'$ indicates that $\pi$ is finer than $\pi'$, \ie that each block of $\pi'$ is a union of blocks of $\pi$.
\begin{example} \rm For instance,
 \begin{align*}
 \Phi_{\{\{1,4\}\{2,5,6\}\{3,7\}\}} &= M_{\{\{1,4\}\{2,5,6\}\{3,7\}\}} + M_{\{\{1,2,4,5,6\}\{3,7\}\}} + M_{\{\{1,3,4,7\}\{2,5,6\}\}}\\
 &\quad + M_{\{\{1,4\}\{2,3,5,6,7\}\}} + M_{\{\{1,2,3,4,5,6,7\}\}}.
 \end{align*}
 \end{example}

From the definition of the $M_{\pi}$, we deduce that the polynomial representation of the word monomial functions is given by $M_\pi(\A)=\sum_{\mathtt w}\mathtt w$ where the sum is over the words $\mathtt w=\mathtt a_1\cdots \mathtt a_n$ where $i,j\in \pi_\ell$ if and only if $\mathtt a_i=\mathtt a_j$, where $\pi=\{\pi_1,\dots,\pi_k\}$ is a set partition of $\{1,\dots,n\}$.
 \begin{example}
 $
 M_{\{\{1,4\}\{2,5,6\}\{3,7\}\}}(\A) = \sum_{\substack{\mathtt a,\mathtt b,\mathtt c \in \A\\\mathtt a \neq\mathtt  b,\mathtt a \neq\mathtt  c,\mathtt b \neq \mathtt c}} \mathtt a\mathtt b\mathtt c\mathtt a\mathtt b\mathtt b\mathtt c.
 $
 \end{example}
The analogue of complete symmetric functions is the basis $(S_\pi)_\pi$ of $\PiQSym$ which is the dual of the basis $(M_\pi)_\pi$ of ${\bf WSym}$. 

The algebra $\PiQSym$ is also realized in the space ${\bf WSym}(\A)$: it is the subalgebra of $(\C\langle \A\rangle,\shuffle)$ generated by $\Psi_{\pi}(\A)=\pi!\Phi_\pi(\A)$ where $\pi!=\#\pi_1!\cdots\#\pi_k!$ for $\pi=\{\pi_1,\dots,\pi_k\}$. Indeed, the linear map $\Psi_\pi\longrightarrow\Psi_\pi(\A)$ is a bijection sending $\Psi_{\pi_1}\Psi_{\pi_2}$ to
\[\begin{array}{rcl}\displaystyle
\sum_{\pi=\pi'_1\cup\pi'_2,\ \pi'_1\cap\pi'_2=\emptyset \atop \pi_1= \std(\pi'_1),\ \pi_2= \std(\pi'_2)}\Psi_\pi(\A)&=&\displaystyle\pi_1!\pi_2!\sum_{\pi=\pi'_1\cup\pi'_2,\ \pi'_1\cap\pi'_2=\emptyset \atop \pi_1= \std(\pi'_1),\ \pi_2= \std(\pi'_2)}\Phi_\pi(\A)\\
&=&\pi_1!\pi_2!\Phi_{\pi_1}(\A)\shuffle\Phi_{\pi_2}(\A)=\Psi_{\pi_1}(\A)\shuffle
\Psi_{\pi_2}(\A).
\end{array}\]

With these notations the image of $S_\pi$ is $S_\pi(\A)=\sum_{\pi'\leq\pi}\Psi_{\pi'}(\A)$. 
For our realization, the duality bracket $\langle\ |\ \rangle$ implements the scalar product $\langle\ |\ \rangle$ on the space ${\bf WSym}(\A)$ for which $\langle S_{\pi_1}(\A)|M_{\pi_2}(\A)\rangle=\langle \Phi_{\pi_1}(\A)|\Psi_{\pi_2}(\A)\rangle=\delta_{\pi_1,\pi_2}$.

The subalgebra of $({\bf WSym}(\A),\shuffle)$ generated by the complete functions $S_{\{\{1,\dots,n\}\}}(\A)$ is isomorphic to $Sym$. Therefore, we define $\sigma_t^W(\mathbb{A})$ and $\phi_t^W(\mathbb{A})$ by
\[
\sigma_t^W(\mathbb{A})=\sum_{n\geq 0}S_{\{\{1,\dots,n\}\}}(\mathbb{A})t^n
\]
and
\[
\phi_t^W(\mathbb{A})=\sum_{n\geq 1}\Psi^{\{\{1,\dots,n\}\}}(\mathbb{A})t^{n-1}.
\]
These series are linked by the equality
\begin{equation}\label{WNewton}
\sigma_t^W(\mathbb{A})= \textrm{exp}_{\shuffle} \left(\phi_t^W(\mathbb{A})\right),
\end{equation}
where $\textrm{exp}_{\shuffle}$ is the exponential in $({\bf WSym}(\mathbb{A}),\shuffle)$.
Furthermore, the coproduct of ${\bf WSym}$ consists in identifying the algebra ${\bf WSym}\otimes{\bf WSym}$ with ${\bf WSym}(\A+\B)$, where $\A$ and $\B$ are two alphabets such that the letters of $\A$ commute with those of $\B$. Hence, one has
$
\sigma_t^W(\mathbb{A}+\mathbb{B})=\sigma_t^W(\mathbb{A})\shuffle\sigma_t^W(\mathbb{B}).
$
In particular, we define the multiplication of an alphabet $\A$ by a constant $k\in\N$ by
\[
\sigma_t^W(k\mathbb A)=\sum_{n\geq 0}S_{\{\{1,\dots,n\}\}}(k\mathbb{A})t^n=\sigma_t^W(\mathbb A)^k.
\]
Nevertheless, the notion of specialization is subtler to define than in $Sym$. Indeed, the knowledge of the complete functions $S_{\{\{1,\dots,n\}\}}(\A)$ does not allow us to recover all the polynomials using uniquely the algebraic operations. In \cite{bultel2013redfield}, we made an attempt to define virtual alphabets by reconstituting the whole algebra using the action of an operad. Although the general mechanism remains to be defined, the case where each complete function $S_{\{\{1,\dots,n\}\}}(\A)$ is specialized to a sum of words of length $n$ can be understood via this construction. More precisely, we consider the family of multilinear $k$-ary operators $\shuffle_\Pi$ indexed by set compositions (a set composition is a sequence $[\pi_1,\dots,\pi_k]$ of subsets of $\{1,\dots,n\}$ such that $\{\pi_1,\dots,\pi_k\}$ is a set partition of $\{1,\dots,n\}$) 
acting on words by $\shuffle_{[\pi_1,\dots,\pi_k]}(a^1_1\cdots a^{1}_{n_1},\dots,a^k_1\cdots a^{k}_{n_k})=b_1\cdots b_n$ with $b_{i^p_\ell}=a^p_\ell$ if $\pi_p=\{i^p_1<\cdots< i^p_{n_p}\}$ and $\shuffle_{[\pi_1,\dots,\pi_k]}(a^1_1\cdots a^{1}_{n_1},\dots,a^k_1\cdots a^{k}_{n_k})=0$ if $\#\pi_p\neq n_p$ for some $1\leq p\leq k$. 

Let $P=(P_n)_{n \geq 1}$ be a family of a homogeneous word polynomials such that $\deg(P_n)=n$ for each $n$. We set $S_{\{\{1,\dots,n\}\}}\left[\mathbb A^{(P)}\right]=P_n$ and $$S_{\{\pi_1,\dots,\pi_k\}} \left[\mathbb A^{(P)}\right]=\shuffle_{[\pi_1,\dots,\pi_k]}(S_{\{\{1,\dots,\#\pi_1\}\}}\left[\A^{(P)}\right],\dots,S_{\{\{1,\dots,\#\pi_k\}\}}\left[\A^{(P)}\right]).$$
 The space ${\bf WSym}\left[\A^{(P)}\right]$ generated by the polynomials $S_{\{\pi_1,\dots,\pi_k\}}\left[\A^{(P)}\right]$ and endowed with the two products $\cdot$ and $\shuffle$ is homomorphic to the double algebra $({\bf WSym}(\A),\cdot,\shuffle)$. 
Indeed, let $\pi=\{\pi_1,\dots,\pi_k\}\vDash n$ and  $\pi'=\{\pi'_1,\dots,\pi'_{k'}\}\vDash n'$ be two set partitions, one has

 \begin{align*}S_{\pi}\left[\A^{(P)}\right]\cdot S_{\pi'}\left[\A^{(P)}\right]&=\shuffle_{[\{1,\dots,n\},\{n+1,\dots,n+n'\}]}\left(S_{\pi}\left[\A^{(P)}\right], S_{\pi'}\left[\A^{(P)}\right]\right)
\\&=\shuffle_{[\pi_1,\dots,\pi_k,\pi'_1[n],\dots,\pi'_{k'}[n]]}\left(S_{\{1,\dots,\#\pi_1\}}\left[\A^{(P)}\right],\dots, 
S_{\{1,\dots,\#\pi_k\}}\left[\A^{(P)}\right],\right.\\&\qquad\left.S_{\{1,\dots,\#\pi'_1\}}\left[\A^{(P)}\right],\dots,S_{\{1,\dots,\#\pi'_{k'}\}}\left[\A^{(P)}\right]\right)\\
&= S_{\pi\uplus\pi'}\left[\A^{(P)}\right]
\end{align*} 
 and
\begin{align*}
 S_{\pi}\left[\A^{(P)}\right]\shuffle S_{\pi'}\left[\A^{(P)}\right]&=\displaystyle\sum_{I\cup J=\{1,\dots,n+n'\},\ I\cap J=\emptyset}\shuffle_{[I,J]}\left(S_{\pi}\left[\A^{(P)}\right],S_{\pi'}\left[\A^{(P)}\right]\right)\\
 &=\displaystyle\sum \shuffle_{[\pi''_1,\dots,\pi''_{k+k'}]}
\left(S_{\{1,\dots,\#\pi_1\}}\left[\A^{(P)}\right],\dots, 
S_{\{1,\dots,\#\pi_k\}}\left[\A^{(P)}\right],\right.\\&\qquad\left.S_{\{1,\dots,\#\pi'_1\}}\left[\A^{(P)}\right],\dots,S_{\{1,\dots,\#\pi'_{k'}\}}\left(\A^{(P)}\right]\right), 
 \end{align*}
 where the second sum is over the partitions $\{\pi''_1,\dots,\pi''_{k+k'}\}\in\pi\Cup\pi'$ satisfying, for each $k+1\leq i\leq k+k'$, $\std(\{\pi''_1,\dots,\pi''_{k}\})=\pi$,  $\std(\{\pi''_{k+1},\dots,\pi''_{k+k'}\})=\pi'$, $\#\pi''_i=\pi_i$. 
Hence, 
 $$ S_{\pi}\left[\A^{(P)}\right]\shuffle S_{\pi'}\left[\A^{(P)}\right]=\sum_{\pi''\in\pi\Cup\pi'}S_{\pi''}\left[\A^{(P)}\right].$$
 In other words, we consider the elements of ${\bf WSym}\left[\A^{(P)}\right]$ as word polynomials in the virtual alphabet $\A^{(P)}$  specializing the elements of ${\bf WSym}(\A)$.
\subsubsection{Biword symmetric functions}

The bi-indexed word algebra ${\bf BWSym}$ was defined in \cite{bultel2013redfield}. We recall its definition here:
the bases of ${\bf BWSym}$  are indexed by set partitions into lists, which can be constructed from a 
set partition by ordering each
block. We will denote by $\mathcal{PL}_n$ the set of the set partitions of $\{1,\dots,n\}$ into lists.

\begin{example}\rm$\{[1, 2, 3], [4, 5]\}$ and $\{[3, 1, 2], [5, 4]\}$ are two distinct set partitions into lists of the set
$\{1, 2, 3, 4, 5\}$.
\end{example}
 
The number of set partitions into lists of an $n$-element set (or set partitions into lists of size
$n$) is given by Sloane’s sequence \seqnum{A000262} \cite{Sloane}. The first values are
\[ 1, 1, 3, 13, 73, 501, 4051, \dots \]
If $\hat \Pi$ is a set partition into lists of $\{1,\dots,n\}$, we will write $\Pi\Vvdash n$.
Set $\hat\Pi \uplus \hat\Pi' = \hat\Pi \cup \{[l_1+n,\dots,l_k+n]:[l_1,\dots,l_k]\in\hat\Pi'\}\Vvdash n+n'$.
   Let $\hat\Pi'\subset\hat\Pi\Vvdash n$, since the integers appearing in $\hat\Pi'$ are all distinct, 
   the standardized $\std(\hat\Pi')$ of $\hat\Pi'$ is well defined as
    the unique set partition into lists obtained 
    by replacing the $i$th smallest integer in $\hat\Pi$ by $i$. For example,
$\std(\{[5,2],[3,10],[6,8]\}) = \{[3,1],[2,6],[4,5]\}.$ 

The Hopf algebra ${\bf BWSym}$ is formally defined by its basis
 $(\Phi_{\hat\Pi})$ where the $\hat\Pi$ are set partitions into lists, its product
$
\Phi_{\hat\Pi} \Phi_{\hat\Pi'} = \Phi_{\hat\Pi\uplus \hat\Pi'}
$
and its coproduct
\begin{equation}\label{copphibw}
\Delta(\Phi_{\hat\Pi}) = \sum^\wedge \Phi_{\std(\hat\Pi')}\otimes\Phi_{\std(\hat\Pi'')},
\end{equation}
where the $\displaystyle \sum^\wedge$ means that the sum is over the $(\hat\Pi',\hat\Pi'')$ such that $\hat\Pi'\cup \hat\Pi''=\hat\Pi$ and $\hat\Pi'\cap\hat\Pi'' =
 \emptyset$.

The product of the graded dual $\BPiQSym$ of ${\bf BWSym}$ is completely described in the dual basis $(\Psi_{\hat\Pi})_{\hat\Pi}$  of $(\Phi_{\hat\Pi})_{\hat\Pi}$ by
\[
\Psi_{\hat\Pi_1}\Psi_{\hat\Pi_2}=\sum_{\hat\Pi=\hat\Pi'_1\cup\hat\Pi'_2,\ \hat\Pi'_1\cap\hat\Pi'_2\atop \std(\hat\Pi'_1)=\hat\Pi_1,\ \std(\hat\Pi'_1)=\hat\Pi_1}\Psi_{\hat\Pi}.
\]

Now consider a bijection $\iota_n$ from $\{1,\dots,n\}$ to the symmetric group $\S_n$. The linear map $\kappa:\mathcal{CP}(1!,2!,3!,\dots)\longrightarrow \mathcal{CL}$ sending $$\{[\{i_1^1,\dots,i_{n_1}^1\},m_1],\dots,[\{i_1^k,\dots,i_{n_k}^k\},m_1]\}\in\mathcal{CP}_n(1!,2!,3!,\dots),$$ with $i_1^j\leq\cdots\leq i_{n_j}^j$, to $$\{[i_{\iota_{n_1}(1)}^1,\dots,i_{\iota_{n_1}(n_1)}^1],\dots,[i_{\iota_{n_\ell}(1)}^\ell,\dots,i_{\iota_{n_\ell}(n_\ell)}^\ell]\}$$ is a bijection. Hence, a fast checking shows that the linear map sending $\Psi^{\Pi}$ to $\Psi^{\kappa(\Pi)}$ is an isomorphism. So, we have
\begin{proposition}~ 
\begin{itemize}
\item The Hopf algebras ${\bf CWSym}(1!,2!,3!,\dots)$ and ${\bf BWSym}$ are isomorphic.
\item The Hopf algebras $\CPiQSym(1!,2!,3!,\dots)$ and $\BPiQSym$ are isomorphic.
\end{itemize}
\end{proposition}

\subsubsection{Word symmetric functions of level $2$}
We consider the algebra ${\bf WSym}_{(2)}$ which is spanned by the $\Phi_\Pi$ where $\Pi$ is a set partition of level 2, that is, a partition  of a partition $\pi$ of $\{1,\dots,n\}$ for some $n$.
The product of this algebra is given by
$
\Phi_{\Pi}\Phi_{\Pi'}=\Phi_{\Pi\cup\Pi'[n]}
$
where $\Pi'[n]=\{e[n]:e\in\Pi'\}$.
The dimensions of this algebra are given by the exponential generating function
$$
\sum_i b^{(2)}_i{t^i\over i!}=\exp(\exp(\exp(t)-1)-1).
$$
The first values are 
$$
1, 3, 12, 60, 358, 2471, 19302, 167894, 1606137, \dots
$$
see sequence \seqnum{A000258} of \cite{Sloane}.

The coproduct is defined by
$
\Delta(\Phi_{\Pi})=\sum_{\Pi'\cup\Pi''=\Pi\atop \Pi'\cap\Pi''=\emptyset} \Phi_{\std(\Pi')}\otimes\Phi_{\std(\Pi'')}
$
where, if $\Pi$ is a partition of a partition of $\{i_1,\dots,i_k\}$,  $\std(\Pi)$ denotes the standardized of $\Pi$, that is the partition of partition of $\{1,\dots,k\}$ obtained 
by substituting each occurrence of $i_j$ by $j$ in $\Pi$. The coproduct being co-commutative, the dual algebra $\PiQSym_{(2)}:={\bf WSym}_{(2)}^*$ is commutative.
The algebra $\PiQSym_{(2)}$ is spanned by a basis $\left(\Psi_\Pi\right)_{\Pi}$ satisfying
$
\Psi_\Pi \Psi_{\Pi'}=\sum_{\Pi''}C_{\Pi,\Pi'}^{\Pi''}\Psi_{\Pi''}
$
where $C_{\Pi,\Pi'}^{\Pi''}$ is the number of ways to write $\Pi''=A\cup B$ with 
$A\cap B=\emptyset$, $\std(A)=\Pi$ and $\std(B)=\Pi'$.

Let $b_n$ be the $n$th Bell number $A_n(1,1,\dots)$.
Considering a bijection from $\{1,\dots,b_n\}$ to the set of the set partitions of $\{1,\dots,n\}$ for each $n$, we obtain, in the same way as in the previous subsection, the following result
\begin{proposition}~ 
\begin{itemize}
\item The Hopf algebras ${\bf CWSym}(b_1,b_2,b_3,\dots)$ and ${\bf WSym}_{(2)}$ are isomorphic.
\item The Hopf algebras $\CPiQSym(b_1,b_2,b_3,\dots)$ and $\PiQSym_{(2)}$ are isomorphic.
\end{itemize}
\end{proposition}
\subsubsection{Cycle word symmetric functions}\label{CWSymF}
We consider the Grossman-Larson Hopf algebra of heap-ordered trees ${\bf \mathfrak{S} Sym}$ \cite{grossman1989hopf}. The combinatorics of this algebra has been extensively investigated in \cite{hivert2008commutative}. This Hopf algebra is spanned by the $\Phi_\sigma$ where $\sigma$ is a permutation. We identify each permutation with the set of its cycles (for example, the permutation $321$ is $\{(13),(2)\}$). The product in this algebra is given by $\Phi_\sigma\Phi_\tau = \Phi_{\sigma\cup\tau[n]}$, where $n$ is the size of the permutation $\sigma$ and $\tau[n]=\{(i_1+n,i_2+n,\ldots,i_k+n)\mid (i_1,\ldots,i_k)\in\tau\}$. The coproduct is given by\begin{equation}
\Delta(\Phi_\sigma) = \sum \Phi_{{\rm std}(\sigma|_{I})}\otimes\Phi_{{\rm std}(\sigma|_J)},
\end{equation}
where the sum is over the partitions of $\{1,\dots,n\}$ into $2$ sets $I$ and $J$ such that the action of $\sigma$ lets the sets $I$ and $J$ globally invariant, $\sigma|_{I}$ denotes the restriction of the permutation $\sigma$ to the set $I$ and 
 ${\rm std}(\sigma|_I)$ is the permutation obtained from $\sigma|_I$ by replacing the $i$th smallest label by $i$ in $\sigma|_I$.
\begin{example}\rm
\[
\Delta(\Phi_{3241})=\Phi_{3241}\otimes 1+\Phi_1\otimes\Phi_{231}+\Phi_{231}\otimes\Phi_1+1\otimes\Phi_{3241}.
\]
\end{example}
 The basis $(\Phi_\sigma)$ and its dual basis $(\Psi_\sigma)$ are respectively denoted by $(S^\sigma)$  and $(M_\sigma)$ in \cite{hivert2008commutative}. The Hopf algebra ${\bf \mathfrak{S} Sym}$ is not commutative but it is cocommutative, so it is not autodual and not isomorphic to the Hopf algebra of free quasi-symmetric functions.\\
Let $\iota_n$ be a bijection from the set of the cycles of $\S_n$ to $\{1,\dots,(n-1)!\}$. We define the bijection $\kappa:\S_n\leftrightarrow \mathcal{CP}(0!,1!,2!,\dots)$ by
$$\kappa(\sigma)=\{[\mathrm{support}(c_1),\iota_{\#\mathrm{support}(c_1)}(\std(c_1))],\dots,[\mathrm{support}(c_k),\iota_{\#\mathrm{support}(c_k)}(\std(c_k))]\},$$ if $\sigma=c_1\dots c_k$ is the decomposition of $\sigma$ into cycles and $\mathrm{support}(c)$ denotes the support of the cycle $c$, \emph{i.e.} the set of the elements which are permuted by the cycle.
\begin{example}
\rm For instance, set
\[\iota_1(1)=1,\ \iota_3(231)=2,\mbox{ and }\iota_3(312)=1.\]
One has
\[
\kappa(32415867)=\{[\{2\},1],[\{1,3,4\},2],[\{5\},1],[\{6,7,8\},1]\}.
\]
\end{example}
The linear map $K:{\bf \mathfrak{S} Sym}\longrightarrow {\bf CWSym}(0!,1!,2!,\ldots)$ sending $\Phi_\sigma$ to $\Phi_{\kappa(\sigma)}$ is an isomorphism of algebra. Indeed, it is straightforward to see that it is a bijection and furthermore $\kappa(\sigma\cup\tau[n])=\kappa(\sigma)\uplus\kappa(\tau)$. Moreover, if $\sigma\in\S_n$ is a permutation and $\{I,J\}$ is a partition of $\{1,\dots,n\}$ into two subsets such that the action of $\sigma$ lets $I$ and $J$ globally invariant, we check that $\kappa(\sigma)=\Pi_1\cup \Pi_2$ with $\Pi_1\cap \Pi_2=\emptyset$, $\std(\Pi_1)=\kappa({\rm std}(\sigma|_{I}))$ and $\std(\Pi_2)=\kappa({\rm std}(\sigma|_{J}))$.
 Conversely, if $\kappa(\sigma)=\Pi_1\cup \Pi_2$ with $\Pi_1\cap \Pi_2=\emptyset$ then there exists   a partition $\{I,J\}$ of $\{1,\dots,n\}$ into two subsets such that the action of $\sigma$ lets $I$ and $J$ globally invariant and $\std(\Pi_1)=\kappa({\rm std}(\sigma|_{I}))$ and $\std(\Pi_2)=\kappa({\rm std}(\sigma|_{J}))$.

 In other words,
$$\Delta(\Phi_{\kappa(\sigma)})=\sum \Phi_{\kappa({\rm std}(\sigma|_{I}))}\otimes\Phi_{\kappa({\rm std}(\sigma|_J)},$$
where the sum is over the partitions of $\{1,\dots,n\}$ into $2$ sets $I$ and $J$ such that the action of $\sigma$ lets the sets $I$ and $J$ globally invariant.  
 Hence $K$ is a morphism of cogebras and, as with the previous examples, one has
\begin{proposition}~ 
\begin{itemize}
\item The Hopf algebras ${\bf CWSym}(0!,1!,2!,\ldots)$ and ${\bf \mathfrak{S} Sym}$ are isomorphic.
\item The Hopf algebras ${\rm C\Pi QSym}(0!,1!,2!,\ldots)$ and ${\bf \mathfrak{S} Sym}^*$ are isomorphic.
\end{itemize}
\end{proposition}

\subsubsection{Miscellanous subalgebras of the Hopf algebra of endofunctions}
We denote by ${\rm End}$ the combinatorial class of endofunctions (an endofunction of size $n\in\mathbb{N}$ is a function from $\{1,\ldots,n\}$ to itself). Given a function $f$ from a finite subset $A$ of $\mathbb{N}$ to itself, we denote by ${\rm std}(f)$ the endofunction $\phi\circ f\circ \phi^{-1}$, where $\phi$ is the unique increasing bijection from $A$ to $\{1,2,\ldots,{\rm card}(A)\}$. Given a function $g$ from a finite subset $B$ of $\mathbb{N}$ (disjoint from A) to itself, we denote by $f\cup g$ the function from $A\cup B$ to itself whose $f$ and $g$ are respectively the restrictions to $A$ and $B$. Finally, given two endofunctions $f$ and $g$, respectively of size $n$ and $m$, we denote by $f\bullet g$ the endofunction $f\cup\tilde{g}$, where $\tilde{g}$ is the unique function from $\{n+1,n+2,\ldots,n+m\}$ to itself such that ${\rm std}(\tilde{g})=g$. \\
Now, let ${\rm EQSym}$ be the Hopf algebra of endofunctions \cite{hivert2008commutative}. This Hopf algebra is defined by its basis $(\Psi_f)$ indexed by endofunctions, the product
\begin{equation}
\Psi_f \Psi_g=\sum_{{\rm std}(\tilde{f})=f,{\rm std}(\tilde{g})=g,\tilde{f}\cup\tilde{g}\in{\rm End}}\Psi_{\tilde{f}\cup\tilde{g}}
\end{equation}
and the coproduct
\begin{equation}
\Delta(\Psi_h)=\sum_{f\bullet g=h}\Psi_f\otimes \Psi_g.
\end{equation}
This algebra is commutative but not cocommutative. We denote by ${\bf ESym}:={\rm EQSym}^*$ its graded dual, and by $(\Phi_f)$ the basis of ${\bf ESym}$ dual to $(\Psi_f)$. The bases $(\Phi_\sigma)$ and $(\Psi_\sigma)$ are respectively denoted by $(S^\sigma)$  and $(M_\sigma)$ in \cite{hivert2008commutative}. The product and the coproduct in ${\bf ESym}$ are respectively given by
\begin{equation}
\Phi_f \Phi_g=\Phi_{f\bullet g}
\end{equation}
and
\begin{equation}
\Delta(\Phi_h)=\sum_{f\cup g=h}\Phi_{{\rm std}(f)}\otimes \Phi_{{\rm std}(g)}.
\end{equation}

\emph{Remark} : The $\Psi_f$, where $f$ is a bijective endofunction, span a Hopf subalgebra of ${\rm EQSym}$ obviously isomorphic to ${\rm \mathfrak{S}QSym} := {\bf\mathfrak{S}Sym}^*$, that is isomorphic to ${\rm C\Pi QSym}(0!,1!,2!, \dots)$ from (\ref{CWSymF}).

As suggested by \cite{hivert2008commutative}, we investigate a few other Hopf subalgebras of ${\rm EQSym}$.
\begin{itemize}
\item[•] The Hopf algebra of idempotent endofunctions is isomorphic to the Hopf algebra ${\rm C\Pi QSym}(1,2,3, \dots)$ . The explicit isomorphism sends $\Psi_f$ to $\Psi_{\phi(f)}$, where for any idempotent endofunction $f$ of size $n$,
\begin{equation}
\phi(f)=\bigg\{\bigg[f^{-1}(i),{\rm card}(\{j\in f^{-1}(i)\mid j\leq i\})\bigg]\bigg| 1\leq i\leq n, f^{-1}(i)\neq \emptyset\bigg\}.
\end{equation} 
\item[•] The Hopf algebra of involutive endofunctions is isomorphic to $${\rm C\Pi QSym}(1,1,0,\ldots,0, \dots) \hookrightarrow  {\rm \Pi QSym}.$$ Namely, it is a Hopf subalgebra of ${\rm \mathfrak{S} QSym}$, and the natural isomorphism from ${\rm \mathfrak{S} QSym}$ to ${\rm C\Pi QSym}(0!,1!,2!,\ldots)$ sends it to the sub algebra ${\rm C\Pi QSym}(1,1,0,\ldots,0, \dots).$
\item[•] In the same way the endofunctions such that $f^3={\rm Id}$ generate a Hopf subalgebra of ${\rm \mathfrak{S}QSym} \hookrightarrow {\rm EQSym} $ isomorphic to the Hopf algebra ${\rm C\Pi QSym}(1,0,2,0, \dots,0,\ldots)$. 
\item[•] More generally, the endofunctions such that $f^p = {\rm Id}$ generate a Hopf subalgebra of ${\rm \mathfrak{S}QSym}\hookrightarrow {\rm EQSym}$ isomorphic to ${\rm C\Pi QSym}(\tau(p))$ where $\tau(p)_i = (i-1)!$ if $i\mid p$ and $\tau(p)_i=0$ otherwise.  
\end{itemize}

\subsection{About specializations}
The aim of this section is to show how the specialization $c_n\longrightarrow {a_n\over n!}$ factorizes through $\PiQSym$ and $\CPiQSym$.

 Notice first that the algebra $Sym$ is isomorphic to the subalgebra of $\PiQSym$ generated by the family $(\Psi_{\{\{1,\dots,n\}\}})_{n\in\N}$; the explicit isomorphism $\alpha$ sends $c_n$ to $\Psi_{\{\{1,\dots,n\}\}}$. The image of $h_n$ is $S_{\{\{1,\dots,n\}\}}$ and the image of $c_\lambda=\frac{1}{\lambda^!}c_{\lambda_1}\cdots c_{\lambda_k}$ is $\sum_{\pi\vDash\lambda}\Psi_\pi$ where $\pi\vDash\lambda$ means that $\pi=\{\pi_1,\dots,\pi_k\}$ is a set partition such that $\#\pi_1=\lambda_1,\ \dots,\#\pi_k=\lambda_k$ and $\lambda^!=\frac{\lambda_1\cdots\lambda_k}{z_\lambda}=\prod_i\mathfrak m_i(\lambda)!$ where $\mathfrak m_i(\lambda)$ denotes the multiplicity of $i$ in $\lambda$. Indeed, $c_\lambda$ is mapped to $\frac{1}{\lambda^!}\Psi_{\{\{1,\dots,\lambda_1\}\}}\cdots \Psi_{\{\{1,\dots,\lambda_k\}\}}$ and $\Psi_{\{\{1,\dots,\lambda_1\}\}}\cdots \Psi_{\{\{1,\dots,\lambda_k\}\}}=\lambda^!\sum_{\pi\vDash\lambda}\Psi_\pi$.

Now the linear map $\beta_{  a}:\PiQSym\longrightarrow\CPiQSym(  a)$ sending each $\Psi_{\pi}$ to the element $\displaystyle\sum_{\Pi\Rrightarrow\pi}\Psi_{\Pi}$ is a morphism of algebra and the subalgebra $\widetilde{\PiQSym}:=\beta_{  a}(\PiQSym)$ is isomorphic to $\PiQSym$ if and only if $  a\in (\N\setminus\{0\})^\N$.

Let $\gamma_{  a}: \CPiQSym(  a)\longrightarrow\C$ be the linear map sending $\Psi_\Pi$ to $\frac1{|\Pi|!}$. We have 
\[
\gamma_{  a}(\Psi_{\Pi_1}\Psi_{\Pi_2})=\sum_{\Pi=\Pi'_1\cup\Pi'_2,\Pi'_1\cap\Pi'_2=\emptyset
\atop \std(\Pi'_1)=\Pi_1, \std(\Pi'_2)=\Pi_2}\gamma_{  a}(\Psi_{\Pi}).\]
From any subset $A$ of $\{1,2,\ldots,|\Pi_1|+|\Pi_2|\}$ of cardinality $n$, one has $\std(\Pi_1') = \Pi_1$, $\std(\Pi_2')=\Pi_2$ and $\Pi_1'\cap\Pi_2' = \emptyset$, where $\Pi_1'$ is obtained from $\Pi_1$ by replacing each label $i$ by the $i$th smallest element of $A$, and $\Pi_2'$ is obtained from $\Pi_2$ by replacing each label $i$ by the $i$th smallest element of $\{1,\ldots,|\Pi_1|+|\Pi_2|\}\setminus A$. Since there are $\left(\begin{array}{c}|\Pi_1|+|\Pi_2|\\|\Pi_1|\end{array}\right)$ ways to construct $A$, one has

\[
\gamma_{  a}(\Psi_{\Pi_1}\Psi_{\Pi_2})=\frac1{(|\Pi_1|+|\Pi_2|)!}\binom{|\Pi_1|+|\Pi_2|}{|\Pi_1|}=\frac{1}{|\Pi_1|!|\Pi_2|!} =\gamma_{  a}(\Psi_{\Pi_1})\gamma_{  a}(\Psi_{\Pi_2}).
\]

In other words, $\gamma_{  a}$ is a morphism of algebra. Furthermore, the restriction $\hat\gamma_{  a}$ of $\gamma_{  a}$ to $\widetilde{\PiQSym}$ is a morphism of algebra that sends $\beta_{  a}(\Psi_{\{\{1,\dots,n\}\}})$ to $a_{n}\over n!$. It follows that if $f\in Sym$, then one has 
\begin{equation}
f(\X^{(  a)})=\tilde{\gamma}_{  a}(\beta_{  a}(\alpha(f))).
\end{equation}
The following theorem summarizes the section:
\begin{theorem}\label{diagram}
The diagram
\[
\xymatrix{
  {\CPiQSym(  a)} \ar@{<-^{)}}[r] \ar@{->>}[rd]_{\gamma_{  a}} & \widetilde{\PiQSym} \ar@{<<-}[r]^{\beta_{  a}} \ar@{->>}[d]^{\hat\gamma_{  a} }  & \PiQSym \ar@{<-^{)}}[d]^{\alpha} \\
    & \C=Sym[\X^{(a)}]  & Sym \ar@{->>}[l]
  }
\]
is commutative.
\end{theorem}

\section{Word Bell polynomials}\label{SBell}
\subsection{Bell polynomials in $\PiQSym$}
Since $Sym$ is isomorphic to the subalgebra of $\PiQSym$ generated by the elements $\Psi_{\{\{1,\dots,n\}\}}$, we can compute 
 $A_n(\Psi_{\{\{1\}\}}, \Psi_{\{\{1,2\}\}},\dots,\Psi_{\{\{1,\dots,m\}\}},\dots)$. From (\ref{WNewton}), we have
\begin{equation}\label{PIAn}
A_n(1!\Psi_{\{\{1\}\}}, 2!\Psi_{\{\{1,2\}\}},\dots,m!\Psi_{\{\{1,\dots,m\}\}},\dots)=n!S_{\{\{1,\dots,n\}\}}=n!\sum_{\pi\vDash n}\Psi_\pi.
\end{equation}
Notice that, from the previous section, the image of the Bell polynomial \[A_n(\Psi_{\{\{1\}\}}, \Psi_{\{\{1,2\}\}},\dots,\Psi_{\{\{1,\dots,m\}\}},\dots)\]
by the morphism $\gamma$ sending $\Psi_{\{\{1,\dots,n\}\}}$ to $1\over n!$ is
\[
\gamma(A_n(1!\Psi_{\{\{1\}\}}, 2!\Psi_{\{\{1,2\}\}},\dots,m!\Psi_{\{\{1,\dots,m\}\}},\dots))=b_n=A_n(1,1,\dots).
\]

In the same way, we have
\begin{equation}\label{PIBn}
B_{n,k}(1!\Psi_{\{\{1\}\}}, 2!\Psi_{\{\{1,2\}\}},\dots,m!\Psi_{\{\{1,\dots,m\}\}},\dots)=n!\sum_{\pi\vDash n\atop \#\pi=k}\Psi_\pi.
\end{equation}

If $(F_n)_n$ is a homogeneous family of elements of $\PiQSym$, such that $|F_n|=n$, we define
\begin{equation}\label{PiA}\mathcal A_n(F_1,F_2,\dots)=\frac1{n!}A_n({1!}F_1,{2!}F_2,\dots,{m!}F_m,\dots)
\end{equation}
and
\begin{equation}\label{PiB}\mathcal B_{n,k}(F_1,F_2,\dots)=\frac1{n!}B_{n,k}(1!F_1,{2!}F_2,\dots,{m!}F_m,\dots).\end{equation}

Considering the map $\beta_{  a}\circ\alpha$ as a specialization of $Sym$, the following identities hold in $\CPiQSym(  a)$:
\[
\mathcal A_n\left(\sum_{1\leq i\leq a_1}\Psi_{\{[\{1\},i]\}}, \sum_{1\leq i\leq a_2}\Psi_{\{[\{1,2\},i]\}},\dots,\sum_{1\leq i\leq a_m}\Psi_{\{[\{1,\dots,m\},i]\}},\dots\right)=\sum_{\Pi\vDash n}\Psi_\Pi
\]
and
\[
\mathcal B_{n,k}\left(\sum_{1\leq i\leq a_1}\Psi_{\{[\{1\},i]\}}, \sum_{1\leq i\leq a_2}\Psi_{\{[\{1,2\},i]\}},\dots,\sum_{1\leq i\leq a_m}\Psi_{\{[\{1,\dots,m\},i]\}},\dots\right)=\sum_{\Pi\vDash n\atop \#\Pi=k}\Psi_\Pi.
\]

\begin{example}
\rm In $\BPiQSym\sim \CPiQSym(1!,2!,\dots)$, we have
\[\mathcal B_{n,k}\left(\Psi_{\{[1]\}}, \Psi_{\{[1,2]\}}+\Psi_{\{[2,1]\}},\dots,\sum_{\sigma\in\S_m}\Psi_{\{[\sigma]\}},\dots\right)=\sum_{\hat\Pi\vDash n\atop \#\hat\Pi=k}\Psi_{\hat\Pi},
\]
where the sum on the right is over the set partitions of $\{1,\dots,n\}$ into $k$ lists. Considering the morphism sending $\Psi_{\{[\sigma_1,\dots,\sigma_n]\}}$ to $1\over n!$, Theorem \ref{diagram} allows us to recover
$
B_{n,k}(1!,2!,3!,\dots)=L_{n,k},
$
the number of set partitions of $\{1,\dots,n\}$ into $k$ lists.
\end{example}

\begin{example}
\rm In $\PiQSym_{(2)}\sim \CPiQSym(b_1,b_2,\dots)$, we have
\[\mathcal B_{n,k}\left(\Psi_{\{\{\{1\}\}\}}, \Psi_{\{\{\{1,2\}\}\}}+\Psi_{\{\{\{1\},\{2\}\}\}},\dots,\sum_{\pi\vDash m}\Psi_{\{\pi\}},\dots\right)=\sum_{\Pi\mbox{ \tiny partition of }\pi\vDash n\atop \#\hat\Pi=k}\Psi_{\Pi},
\]
where the sum on the right is over the set partitions of $\{1,\dots,n\}$ of level 2 into $k$ blocks. Considering the morphism sending $\Psi_{\{\pi\}}$ to $1\over n!$ for $\pi\vDash n$, Theorem \ref{diagram} allows us to recover
$
B_{n,k}(b_1,b_2,b_3,\dots)=S^{(2)}_{n,k},
$
the number of set partitions into $k$ sets of a partition of $\{1,\dots,n\}$.
\end{example}

\begin{example}
\rm In $\mathbf{\S Sym}^*\sim \CPiQSym(0!,1!,2!\dots)$, we have
\[\mathcal B_{n,k}\left(\Psi_{[1]}, \Psi_{[2,1]}, \Psi_{[2,3,1]}+\Psi_{[3,1,2]},\dots,\sum_{\sigma\in\S_n\atop \sigma\mathrm{\ is\ a\ cycle}}\Psi_{\{\pi\}},\dots\right)=\sum_{\sigma\in\S_n\atop \sigma\mathrm{\ has\ } k\ \mathrm{cycles}}\Psi_{\sigma},
\]
where the sum on the right is over the permutations of size $n$ having $k$ cycles. Considering the morphism sending $\Psi_{\sigma}$ to $1\over n!$ for $\sigma\in \S_n$, Theorem \ref{diagram} allows us to recover
$
B_{n,k}(0!,1!,2!,\dots)=s_{n,k},
$
the number of permutations of $\mathfrak{S}_n$ having exactly $k$ cycles.
\end{example}

\begin{example}
\rm In the Hopf algebra of idempotent endofunctions, we have
\[\begin{array}{l}\mathcal B_{n,k}\left(\Psi_{f_{1,1}},\Psi_{f_{2,1}}+\Psi_{f_{2,2}},\Psi_{f_{3,1}}+\Psi_{f_{3,2}}+\Psi_{f_{3,3}},\ldots,\sum_{i=1}^n \Psi_{f_{n,i}},\ldots\right)=\\\displaystyle\sum_{|f|=n, {\rm card}(f(\{1,\ldots,n\}))=k}\Psi_f,\end{array}\]
where for $i\geq j\geq 1$, $f_{i,j}$ is the constant endofunction of size $i$ and of image $\{j\}$. Here, the sum on the right is over idempotent endofuctions $f$ of size $n$ such that the cardinality of the image of $f$ is $k$.  Considering the morphism sending $\Psi_f$ to $1\over n!$ for $|f|=n$, Theorem \ref{diagram} allows us to recover that
$B_{n,k}(1,2,3,\dots)$ is the number of these idempotent endofunctions, that is the idempotent number $\binom {n}{k} k^{n-k}$ \cite{harris1967number, tainiter1968generating}. 
\end{example}

\subsection{Bell polynomials in ${\bf WSym}$}

Recursive descriptions of Bell polynomials are given in \cite{ebrahimi2014noncommutative}. In this section we rewrite this result and other ones related to these polynomials in the Hopf algebra of word symmetric functions ${\bf WSym}$. We define the operator $\partial$ acting linearly on the left on ${\bf WSym}$ by
\[
1\partial=0\mbox{ and }\Phi_{\{\pi_1,\dots,\pi_k\}}\partial=\sum_{i=1}^k\Phi_{(\{\pi_1,\dots,\pi_k\}\setminus \pi_i)\cup\{\pi_i\cup\{n+1\}\}}.
\]
In fact, the operator $\partial$ acts on $\Phi_\pi$ almost as the multiplication of $M_{\{\{1\}\}}$ on $M_\pi$. More precisely :
\begin{proposition}\label{PartWSym}
We have:
\[ \partial = \phi^{-1} \circ \mu \circ \phi - \mu ,\] 
where  $\phi$ is the linear operator satisfying $M_\pi\phi=\Phi_\pi$ and $\mu$ is the multiplication by $\Phi_{\{\{1\}\}}$.
\end{proposition}
\begin{example} \rm For instance, one has
$$\begin{array}{rcl}
\Phi_{\{\{1,3\},\{2,4\}\}}\partial & = &
\Phi_{\{\{1,3\},\{2,4\}\}} (\phi^{-1}\mu \phi - \mu)\\
& = &M_{\{\{1,3\},\{2,4\}\}} \mu\phi - \Phi_{\{\{1,3\},\{2,4\},\{5\}\}}\\
& =&(M_{\{\{1,3,5\},\{2,4\}\}} + M_{\{\{1,3\},\{2,4,5\}\}}+ M_{\{\{1,3\},\{2,4\},\{5\}\}}) \phi \\&&- \Phi_{\{\{1,3\},\{2,4\},\{5\}\}} \\
& =&\Phi_{\{\{1,3,5\},\{2,4\}\}} + \Phi_{\{\{1,3\},\{2,4,5\}\}}+ \Phi_{\{\{1,3\},\{2,4\},\{5\}\}} \\&&- \Phi_{\{\{1,3\},\{2,4\},\{5\}\}} \\
& =&\Phi_{\{\{1,3,5\},\{2,4\}\}} + \Phi_{\{\{1,3\},\{2,4,5\}\}}.
\end{array}$$
\end{example}

Following Remark \ref{recpart}, we define recursively the elements $\mathfrak A_n$ of ${\bf WSym}$ as
\begin{equation}
\mathfrak A_0=1,\ \mathfrak A_{n+1}=\mathfrak A_n(\Phi_{\{\{1\}\}}+\partial).
\end{equation}
So we have
\begin{equation}\label{EqRec2}
\mathfrak A_{n}= 1(\Phi_{\{1\}} + \partial)^n.
\end{equation}
Easily, one shows that $\mathfrak A_n$ provides an analogue of complete Bell polynomials in ${\bf WSym}$.
\begin{proposition}\label{WBell}
\[
\mathfrak A_n=\sum_{\pi\vDash n}\Phi_\pi.\]
\end{proposition}

Noticing that the multiplication by $\Phi_{\{\{1\}\}}$ adds one part to each partition, we give the following analogue for partial Bell polynomials.
\begin{proposition}\label{WpartialBell}
If we set
\begin{equation}\label{WB}
\mathfrak B_{n,k}= [t^{k}] 1 (t \Phi_{\{\{1\}\}} + \partial)^n,
\end{equation}
then we have $\mathfrak B_{n,k}=\displaystyle\sum_{\pi\vDash n\atop \#\pi=k}\Phi_{\pi}.$
\end{proposition}
\begin{example}\rm We have
\[\begin{array}{l}
1 (t \Phi_{\{\{1\}\}} + \partial)^4 
= t^4 \Phi_{\{\{1\},\{2\},\{3\},\{4\}\}} + t^3 ( \Phi_{\{\{1,2\},\{3\},\{4\}\}} + \Phi_{\{\{1,3\},\{2\},\{4\}\}}\\ 
+ \Phi_{\{\{1\},\{2,3\},\{4\}\}}  + \Phi_{\{\{1,4\},\{2\},\{3\}\}} + \Phi_{\{\{1\},\{2,4\},\{3\}\}}  +  \Phi_{\{\{1\},\{2\},\{3,4\}\}} )  
+ \\
t^2 (\Phi_{\{\{1,3,4\},\{2\}\}} + \Phi_{\{\{1,2,3\},\{4\}\}} + \Phi_{\{\{1,2,4\},\{3\}\}} + \Phi_{\{\{1,2\},\{3,4\}\}} + \Phi_{\{\{1,3\},\{2,4\}\}}\\  +  \Phi_{\{\{1,4\},\{2,3\}\}} + \Phi_{\{\{1\},\{2,3,4\}\}}) + t \Phi_{\{\{1\},\{3\},\{4\},\{2\}\}}.\end{array}
\]
Hence,
$$\begin{array}{rcl}\mathfrak B_{4,2}&=&\Phi_{\{\{1,3,4\},\{2\}\}} + \Phi_{\{\{1,2,3\},\{4\}\}} + \Phi_{\{\{1,2,4\},\{3\}\}} + \Phi_{\{\{1,2\},\{3,4\}\}} \\&&+ \Phi_{\{\{1,3\},\{2,4\}\}} + \Phi_{\{\{1,4\},\{2,3\}\}} + \Phi_{\{\{1\},\{2,3,4\}\}}.\end{array}$$
\end{example}   

\subsection{Bell polynomials in $\mathbb C\langle \mathbb A\rangle$}

Both ${\bf WSym}$ and $\PiQSym$ admit word polynomial realizations in a subspace ${\bf WSym}(\A)$ of the free associative algebra $\mathbb C\langle\mathbb A\rangle$ over an infinite alphabet $\A$. When endowed with the concatenation product, ${\bf WSym}(\A)$ is isomorphic to ${\bf WSym}$ and when endowed with the shuffle product, it is isomorphic to $\PiQSym$. Alternatively to the definitions of partial Bell numbers in $\PiQSym$ (\ref{PiB}) and in ${\bf WSym}$, we  set, for any sequence of  polynomials $(F_i)_{i\in\N}$ in $\C\langle \A\rangle$,
\begin{equation}\label{wordB}
\sum_{n\geq 0}\mathtt B_{n,k}(F_1,\dots,F_m,\dots)t^n=\frac1{k!}\left(\sum_i F_it^i\right)^{\shuffle k}
\end{equation}
and 
\begin{equation}\label{wordA}
\mathtt A_n(F_1,\dots,F_m,\dots)=\sum_{k\geq 1}\mathtt B_{n,k}(F_1,\dots,F_m,\dots).
\end{equation}

This definition generalizes  (\ref{PiB}) and (\ref{WB}) in the following sense:
\begin{proposition}
We have
$$\mathtt B_{n,k}(\Psi_{\{\{1\}\}}(\A),\dots,\Psi_{\{\{1,\dots,m\}\}}(\A),\dots)=\mathcal B_{n,k}
(\Psi_{\{\{1\}\}}(\A),\dots,\Psi_{\{\{1,\dots,m\}\}}(\A),\dots)$$
and
$$\mathtt B_{n,k}(\Phi_{\{\{1\}\}}(\A),\dots,\Phi_{\{\{1,\dots,m\}\}}(\A),\dots)=\mathfrak B_{n,k}
(\A).$$
\end{proposition}
\begin{proof}
The two identities follow from 
\[
\Psi_{\pi_1}(\A)\shuffle\Psi_{\pi_2}(\A)=\sum_{\pi=\pi'_1\cup\pi'_2,\ \pi'_1\cap\pi'_2=\emptyset\atop
\std(\pi'_1)=\pi_1,\ \std(\pi'_2)=\pi_2}\Psi_{\pi}(\A).
\]
\end{proof}

Equality (\ref{wordB}) allows us to show more general properties. For instance, let $\A'$ and $\A''$ be two disjoint subalphabets of $\A$ and set
\[
S^{\A'}_n(\A'')=S_{\{\{1\}\}} (\A')\shuffle S_{\{\{1,\dots,n-1\}\}}(\A'').
\]
Remarking that
\[\begin{array}{l}
\displaystyle\sum_{n}\mathtt B_{n,k}(S^{\A'}_1(\A''),\dots,S^{\A'}_m(\A''),\dots)t^n=\\\displaystyle t^kS_{\{\{1\},\dots,\{k\}\}}(\A')\shuffle(\sum_{n\geq 0}S_{\{\{1,\dots,n\}\}}(\A'')t^n)^{\shuffle k}
=\displaystyle t^kS_{\{\{1\},\dots,\{k\}\}}(\A')\shuffle\sigma^W_t(k\A''),
\end{array}
\]
we obtain a word analogue of the formula allowing one to write a Bell polynomial as a symmetric function (see eq. (\ref{B2hk}) in Appendix \ref{Symmetric}):
\begin{proposition}\label{WordB2Sk}
\[\mathtt B_{n,k}(S^{\A'}_1(\A''),\dots,S^{\A'}_m(\A''),\dots)=S_{\{\{1\},\dots,\{k\}\}}(\A')\shuffle S_{\{\{1,\dots,n-k\}\}}(k\A'').\]
\end{proposition}

For simplicity, let us write $\mathtt B_{n,k}^{\A'}(\A''):=\mathtt B_{n,k}(S^{\A'}_1(\A''),\dots,S^{\A'}_m(\A''),\dots).$
 
 Let $k=k_1+k_2$. 
From 
$$S_{\{\{1\},\dots,\{k_1\}\}}(\A')\shuffle S_{\{\{1\},\dots,\{k_2\}\}}(\A')=\binom k{k_1}S_{\{\{1\},\dots,\{k\}\}}(\A')$$ 
and
$$S_{\{\{1,\dots,n-k\}\}}(k\A'')=\sum_{i+j=n-k}S_{\{\{1,\dots,i\}\}}(k_1\A'')\shuffle S_{\{\{1,\dots,i\}\}}(k_2\A''),$$ 
we deduce an analogue of the binomiality of the partial Bell polynomials (see eq. (\ref{BellareBinomial}) in Appendix \ref{Symmetric}):
\begin{corollary} Let $k=k_1+k_2$ be three nonnegative integers. We have
\begin{equation}\label{k_1+k_2}
{\binom k{k_1}}{\mathtt B}_{n,k}^{\A'}(\A'')=
\sum_{i=0}^{n} {\mathtt B}_{i,k_1}^{\A'}(\A'')\shuffle{\mathtt B}_{n-i,k_2}^{\A'}(\A'').
\end{equation}
\end{corollary}
\begin{example}
\rm Consider a family of functions $\left(f_k\right)_k$ such that $f_k:\N\longrightarrow \C\langle \A\rangle$ satisfying 
\begin{equation}\label{WBinom}f_0=1\mbox{ and } f_n(\alpha+\beta)=\sum_{n=i+j}f_{i}(\alpha)\shuffle f_{j}(\beta).  \end{equation}

From (\ref{wordB}), we obtain
\[
\mathtt B_{n,k}(f_0(a),\dots,f_{m-1}(a),\dots)t^n=
\displaystyle\frac1{k!}\sum_{i_1+\cdots+i_k=n-k}f_{i_1}(a)\shuffle\cdots\shuffle f_{i_k}(a).
\]
Hence, iterating (\ref{WBinom}), we deduce
$$
\mathtt B_{n,k}(f_0(a),\dots,f_{m-1}(a),\dots)=\displaystyle\frac1{k!}f_{n-k}(ka).
$$
Set $f_n(k)=k!\mathtt B^{\A'}_{n,k}(\A'')$ and $f_0(k)=1$. By (\ref{k_1+k_2}), the family $(f_n)_{n\in\N}$ satisfies (\ref{WBinom}). Hence we obtain an analogue of composition formula (see eq. (\ref{BofB}) in  Appendix \ref{Symmetric}):
$$
 k_1!\mathtt B_{n,k_1}(1,\dots,k_2!\mathtt B^{\A'}_{m-1,k_2}(\A''),\dots)={(k_1k_2)!}\mathtt B^{\A'}_{n-k_1,k_1k_2}(\A'').
$$
\end{example}
Suppose now $\A''=\A''_1+\A''_2$.\\
By
$
S_{\{\{1,\dots,n\}\}}(\A'')=\sum_{i=0}^nS_{\{\{1,\dots,i\}\}}(\A''_1)\shuffle S_{\{\{1,\dots,n-i\}\}}(\A''_2),
$
Proposition (\ref{WordB2Sk}) allows us to write a word analogue of the convolution formula for Bell polynomials (see formula (\ref{convolutionBell}) in Appendix \ref{Symmetric}):
\begin{corollary}
\begin{equation}\label{A1+A2}
S_{\{\{1\},\dots,\{k\}\}}(\A')\shuffle\mathtt B_{n,k}^{\A'}(\A'')= \sum_{i=0}^{n} \mathtt B_{i,k}^{\A'}(\A''_1)\shuffle \mathtt B_{n-i,k}^{\A'}(\A''_2).
\end{equation}
\end{corollary}

Let $k_1$ and $k_2$ be two positive integers. We have
\[\begin{array}{l}\displaystyle
\sum_n\mathtt B_{n,k_1}(\mathtt B^{\A'}_{k_2,k_2}(\A''),\dots,\mathtt B^{\A'}_{k_2+m-1,k_2}(\A''),\dots)t^n=\\\displaystyle
\frac1{k_1!}\left(\sum_{m\geq 1}\mathtt B^{\A'}_{k_2+m-1,k_2}(\A'')t^m\right)^{\shuffle k_1}=
t^{k_1}S_{\{\{1\},\dots,\{k_2k_1\}\}}(\A')\shuffle\sigma_t^W(k_1k_2\A'').\end{array}
\]
Hence,
\begin{proposition}
\begin{equation}\label{CompWB}
\begin{array}{l}
\mathtt B_{n,k_1}(\mathtt B^{\A'}_{k_2,k_2}(\A''),\dots,\mathtt B^{\A'}_{k_2+m-1,k_2}(\A''),\dots)=
\mathtt B^{\A'}_{n-k_1+k_1k_2,k_1k_2}(\A'').
\end{array}
\end{equation}
\end{proposition}

\subsection{Specialization again}

In \cite{bultel2013redfield} we have shown that one can construct a double algebra which is homomorphic to $({\bf WSym}(\A),.,\shuffle)$. 
This is a general construction which is an attempt to define properly the concept of virtual alphabet for ${\bf WSym}$. In our context the construction is simpler, let us briefly recall it.

Let $  F=(F_\pi(\A))_{\pi}$ be a basis of ${\bf WSym}(\A)$. We will say that $  F$ is \emph{shuffle}-compatible if $$F_{\{\pi_1,\dots,\pi_k\}}(\A)=\shuffle_{[\pi_1,\dots,\pi_k]}\left(F_{\{\{1,\dots,\#\pi_1\}\}}(\A),\dots,
F_{\{\{1,\dots,\#\pi_k\}\}}(\A)\right).$$

 Hence, one has
\[
F_{\pi_1}(\A)\shuffle F_{\pi_2}(\A)=\sum_{\pi=\pi'_1\cup\pi'_2,\pi'_1\cap\pi'_2=\emptyset\atop
\std(\pi'_1)=\pi_1, \std(\pi'_2)=\pi_2} F_\pi\mbox{ and }F_{\pi_1}(\A).F_{\pi_2}(\A)=F_{\pi_1\uplus \pi_2}(\A).
\]

\begin{example}
\rm The bases $(S_\pi(\A))_\pi$, $(\Phi_\pi(\A))_\pi$ and $(\Psi_\pi(\A))_\pi$ are shuffle-compatible but not the basis $(M_\pi(\A))_\pi$.
\end{example}
Straightforwardly, one has:
\begin{claim}
Let $(F_\pi(\A))_{\pi}$ be a \emph{shuffle}-compatible basis of ${\bf WSym}(\A)$. Let $\B$ be another alphabet and  let $  P=(P_k)_{k>0}$ be a family of  noncommutative polynomials of $\C\langle \B\rangle$ such that $\deg P_k=k$. Then, the space spanned by the polynomials \[F_{\{\pi_1,\dots,\pi_k\}}[\mathbb A^{(  P)}_F]:=\shuffle_[\pi_1,\dots,\pi_k]\left(P_{\#\pi_1},\dots,
P_{\#\pi_k}\right)\]
is stable under concatenation and shuffle product in $\C\langle \B\rangle$. So it is a double algebra which is homomorphic to $({\bf WSym}(\A),.,\shuffle)$. We will call ${\bf WSym}[\A^{(  P)}_F]$ this double algebra and $f[\A^{(  P)}_F]$ will denote the image of an element $f\in{\bf WSym}(\A)$ by the morphism ${\bf WSym}(\A)\longrightarrow {\bf WSym}[\A^{(  P)}_F]$ sending $F_{\{\pi_1,\dots,\pi_k\}}$ to $F_{\{\pi_1,\dots,\pi_k\}}[\mathbb A^{(  P)}_F]$.
\end{claim}
With these notations, one has
\[
\mathtt B_{n,k}(P_1,\dots,P_m,\dots)=\mathfrak B_{n,k}[\A^{(  P)}_{\Phi}].
\]
\begin{example}\rm
We define a specialization by setting $$\Phi_{\{\{1,\dots,n\}\}}[\S]=\sum_{\sigma\in\S_n\atop \sigma_1=1}\mathtt b_{\sigma[1]}\dots \mathtt b_{\sigma[n]},$$ where the letters $\mathtt b_i$ belong to an alphabet $\B$. Let $\sigma\in\S_n$ be a permutation and $\sigma=c_1\circ\cdots\circ c_k$ its decomposition into cycles. Each cycle $c^{(i)}$ is denoted by a sequence of integers $(n^{(i)}_1,\dots,n^{(i)}_{\ell_i})$ such that $n_1^{(i)}=\min\{ n_1^{(i)},\dots,n_{\ell_i}^{(i)})$. Let $\widetilde {c^{(i)}}\in \S_{\ell_i}$ be the permutation which is the standardized of the sequence $n^{(i)}_1\dots n^{(i)}_{\ell_i}$. The cycle support of $\sigma$ is the partition $$\mathrm{support}(\sigma)=\{\{n^{(1)}_1,\dots,n^{(1)}_{\ell^{(i)}}\},\dots,\{n^{(k)}_1,\dots,n^{(k)}_{\ell^{(k)}}\}\}.$$ 
We define $w[c^{(i)}]=\mathtt b_{\widetilde{c^{(i)}}[1]}\cdots \mathtt b_{\widetilde{c^{(i)}}[\ell^{(i)}]}$ and $w[\sigma]=\shuffle_{[\pi_1,\dots,\pi_k]}(w[c^{(1)}],\dots,w[c^{(k)}])$ where $\pi_i=\{n^{(i)}_1,\dots,n^{(i)}_{\ell^{(i)}}\}$ for each $1\leq i\leq k$.

For instance, if $\sigma=312654=(132)(46)(5)$ we have $w[(132)]=\mathtt b_1\mathtt b_3\mathtt b_2$, 
$w[(46)]=\mathtt b_1\mathtt b_2$, $w[(5)]=\mathtt b_1$ and $w[\sigma]=\mathtt b_1\mathtt b_3\mathtt b_2\mathtt b_1\mathtt b_1\mathtt b_2$.

So, we have
$$\mathtt B_{n,k}(\Phi_{\{\{1\}\}}[\S],\dots,\Phi_{\{\{1,\dots,m\}\}}[\S],\dots)=\mathfrak B_{n,k}[\S]=\sum_{\pi\vDash n\atop \#\pi=k}\Phi_{\pi}[\S]=\sum_{\sigma\in\S_n\atop 
\#\mathrm{support}(\sigma)=k}w[\sigma].$$

For instance
\begin{align*}
&\mathtt B_{4,2}(\mathtt b_1,\mathtt b_1\mathtt b_2,\mathtt b_1\mathtt b_2\mathtt b_3+
\mathtt b_1\mathtt b_3\mathtt b_2,\mathtt b_1\mathtt b_2\mathtt b_3\mathtt b_4+
\mathtt b_1\mathtt b_3\mathtt b_2\mathtt b_4
+\mathtt b_1\mathtt b_2\mathtt b_4\mathtt b_3+
\mathtt b_1\mathtt b_3\mathtt b_4\mathtt b_2+
\mathtt b_1\mathtt b_4\mathtt b_2\mathtt b_3\\
&\qquad+\mathtt b_1\mathtt b_4\mathtt b_3\mathtt b_2,\dots
)\\
&=
\Phi_{\{\{1\},\{2,3,4\}\}}[\S]
+\Phi_{\{\{2\},\{1,3,4\}\}}[\S]+\Phi_{\{\{3\},\{1,2,4\}\}}[\S]+\Phi_{\{\{4\},\{1,2,3\}\}}[\S]+
 \Phi_{\{\{1,2\},\{3,4\}\}}[\S]\\
&\qquad+\Phi_{\{\{1,3\},\{2,4\}\}}[\S]+\Phi_{\{\{1,4\},\{2,3\}\}}[\S] \\
&=2\mathtt b_1\mathtt b_1\mathtt b_2\mathtt b_3+ 2\mathtt b_1\mathtt b_1\mathtt b_3\mathtt b_2+
\mathtt b_1\mathtt b_2\mathtt b_1\mathtt b_3 +
 \mathtt b_1\mathtt b_3\mathtt b_2\mathtt b_1+
\mathtt b_1\mathtt b_2\mathtt b_3\mathtt b_1+\mathtt b_1\mathtt b_3\mathtt b_2\mathtt b_1
+
\mathtt b_1\mathtt b_2\mathtt b_1\mathtt b_2
+2\mathtt b_1\mathtt b_1\mathtt b_2\mathtt b_2.
\end{align*}

Notice that the sum of the coefficients of the words occurring in the expansion of $\mathfrak B_{n,k}[\S]$ is equal to the Stirling number $s_{n,k}$. Hence, this specialization gives another word analogue of  formula (\ref{stirling1}).
\end{example}

\section{Munthe-Kaas polynomials}\label{Munthe-Kaas}
\subsection{Munthe-Kaas polynomials from $WSym$}
In order to generalize the Runge-Kutta method to integration on manifolds, Munthe-Kaas \cite{MK} 
introduced a noncommutative version of Bell polynomials. We recall here the construction in a slightly different variant adapted
to our notation, the operators acting on the left. Consider an alphabet $\mathbb D=\{d_1,d_2,\dots\}$. 
The algebra $\mathbb C\langle \mathbb D\rangle$ is equipped with the derivation defined by $d_i\partial=d_{i+1}$. 
The noncommutative Munthe-Kaas Bell polynomials are defined by setting $t=1$ in $\mathtt{MB}_{n}(t)=1.(td_1+\partial)^n$. 
The partial noncommutative Bell polynomial $\mathtt{MB}_{n,k}$ is the coefficient of $t^k$ in $\mathtt{MB}_n(t)$.
\begin{example}
\begin{itemize}
\item $\mathtt{MB}_{1}(t)=d_1t$,
\item $\mathtt{MB}_2(t)=d_1^2t^2+d_2t$,
\item $\mathtt{MB}_3(t)=d_1^3t^3+(2d_2d_1+d_1d_2)t^2+d_3t$,
\item $\mathtt{MB}_4(t)=d_1^4t^4+(3d_2d_1^2+2d_1d_2d_1+d_1^2d_2)t^3+(3d_3d_1+3d_2^2+d_1d_3)t^2+d_4t$.
\end{itemize}
\end{example}
We consider the map $\chi$ which sends each set partition $\pi$ to the integer composition $[\mathrm{card}(\pi_1),\dots,\mathrm{card}(\pi_k)]$ if 
$\pi=\{\pi_1,\dots,\pi_k\}$ where $\min(\pi_i)<\min(\pi_{i+1})$ for any $0<i<k$. The linear map $\Xi$ sending $\Phi_\pi$ to $d_{\chi(\pi)[1]}\cdots d_{\chi(\pi)[k]}$ is a morphism of algebra.
Hence, we deduce
\begin{proposition}
\begin{equation}
\Xi\left(\mathfrak B_{n,k}\right)=\mathtt{MB}_{n,k}.
\end{equation}
\end{proposition}
\begin{example}
\[\Xi\left(\mathfrak{B}_{3,2}\right)=\Xi\left(\Phi_{\{\{1\},\{2,3\}\}}+\Phi_{\{\{1,3\},\{2\}\}}+
\Phi_{\{\{1,2\},\{3\}\}}\right)=d_1d_2+2d_2d_1=\mathtt{MB}_{3,2}.\]
\end{example}
We recover a result due to Ebrahimi-Fard \emph{et al.} \cite{ebrahimi2014noncommutative}.
\begin{theorem}
If $j_1+\cdots+j_k=n$, the coefficient of $d_{j_1}\cdots d_{j_k}$ in $\mathtt{MB}_{n,k}$ is equal to the number of partitions 
of $\{1,2, \dots, n\}$  into parts $\pi_1,\dots, \pi_k$ such that $\mathrm{card}(\pi_\ell)=j_\ell$ 
for each $1\leq\ell\leq k$ and $\min(\pi_1)<\cdots<\min(\pi_k)$. 
\end{theorem}
\subsection{Dendriform structure and quasideterminant formula}
The algebra $\PiQSym$ is equipped with a Zinbiel structure. The notion of Zinbiel algebra is due to Loday \cite{Loday}. 
This is an algebra equipped with two nonassociative products $\prec$ and $\succ$ satisfying
\begin{itemize}
\item $(u\prec v)\prec w=u\prec(v\prec w)+u\prec(v\succ w)$,
\item $(u\succ v)\prec w=u\succ(v\prec w)$,
\item $u\succ( v\succ w)=(u\prec v)\succ w+(u\succ v)\succ w$,
\item $u\prec v=v\succ u$.
\end{itemize}
The Zinbiel structure on $\PiQSym$ is defined for any $\pi\vDash n$ and $\pi'\vDash m$ 
by $\Phi_\pi\prec\Phi_{\pi'}=\displaystyle\sum'\Phi^{\pi[I]\cup \pi'[J]}$ (resp. $\Phi_\pi\succ\Phi_{\pi'}=\displaystyle\sum''\Phi^{\pi[I]\cup \pi'[J]}$) where $\displaystyle\sum'$  (resp. $\displaystyle\sum''$) means  the 
sum is over the partitions $I\cup J=\{1,\dots,n+m\}$ (
$I\cap J=\emptyset$) such that $\mathrm{card}(I)=n$, $\mathrm{card}(J)=m$, and $1\in A$ (resp. $1\in B$), and $\Phi[I]$ is obtained
from $\pi$ by substituting each $\ell$ by $i_\ell$ if $I=\{i_1,\dots,i_n\}$ and $i_1<\dots<i_n$. 
See also \cite{NovelliThibon2006,NovelliThibon2007,Foissy2007} for other combinatorial Hopf algebras with a dendriform structure.

We notice that one has
\begin{equation}
\sum_{n}\mathcal B_{n,k}(\Phi_{\{\{1\}\}},\Phi_{\{\{1,2\}\}},\dots)t^n=\left(\sum_i \Phi{\{\{1,\dots,i\}\}}t^i\right)^{\displaystyle\mathop\prec^{\rightarrow}k}
\end{equation}
with $u^{\displaystyle\mathop\prec^{\rightarrow}k}=u^{\displaystyle\mathop\prec^{\rightarrow}k-1}\prec u$ and $u^{\displaystyle\mathop\prec^{\rightarrow}0}=1$.
\begin{definition}
Let $A_n=(a_{ij})_{1\leq i\leq j\leq n}$ be an upper triangular matrix whose entries are in a Zinbiel algebra. We define the polynomial
\begin{equation}
\mathrm P(A_n;t)=t\sum_{k=1}^n\mathrm P(A_{k-1})\prec a_{k,n}\mbox{ and }\mathrm P(A_0)=1.
\end{equation}
\end{definition}
\begin{example}
\rm 
\[
\begin{array}{rcl}
\mathrm P(A_4,t)&=&t\mathrm P(A_3)\prec a_{44}+t\mathrm P(A_2)\prec a_{34}+t\mathrm P(A_1)\prec a_{24}+t\mathrm P(A_0)\prec a_{14}\\
&=&t^4((a_{11}\prec a_{22})\prec a_{33})\prec a_{44}+t^3(a_{11}\prec a_{23})\prec a_{44}\\&&
+t^3(a_{12}\prec a_{33})\prec a_{44}+t^3(a_{11}\prec a_{22})\prec a_{34}+t^2a_{12}\prec a_{34}+at^2_{11}\prec a_{24}\\
&&+ta_{14}.
\end{array}
\]
\end{example}
By induction, we find
\begin{equation}\label{PA_n}
\mathrm P(A_n;t)=ta_{1n}+\sum_{\substack{1\leq j_1<j_2\\<\cdots<j_k}}t^k(\cdots(a_{1j_1}\prec(a_{j_1+1,j_2})\prec(a_{j_2+1,j_3})\prec\cdots \prec
a_{j_{k-1}+1,j_k})\prec a_{j_k+1,n}).
\end{equation}
Setting $M_n:=(\Phi_{\{\{1,\dots,j-i+1\}\}})_{1\leq i\leq j\leq n}$, we find
\begin{proposition}
\begin{equation}\label{B2P}
\mathcal B_{n,k}(\Phi_{\{\{1\}\}},\Phi_{\{\{1,2\}\}},\dots)=[t^k] \mathrm P(M_n;t).
\end{equation}
\end{proposition}
\begin{example}
We have $\mathrm P(A_3;t)=t^3(a_{11}\prec a_{22})\prec a_{33}+t^2(a_{11}\prec a_{23}+a_{12}\prec a_{33})+ta_{13}$. Hence
\begin{align*}
\mathrm P(M_3;t)&=t^3(\Phi_{\{\{1\}\}}\prec \Phi_{\{\{1\}\}})\prec \Phi_{\{\{1\}\}}+
t^2(\Phi_{\{\{1\}\}}\prec \Phi_{\{\{1,2\}\}}+\Phi_{\{\{1,2\}\}})\prec \Phi_{\{\{1\}\}}))\\
&\qquad+t\Phi_{\{\{1,2,3\}\}}\\
&=
t^3\Phi_{\{\{1\},\{2\},\{3\}\}}+t^2(\Phi_{\{\{1\},\{2,3\}}+\Phi_{\{\{1,2\},\{3\}\}}+\Phi_{\{\{1,3\},\{2\}\}})+
t\Phi_{\{\{1,2,3\}\}}\\
&=t^3\mathcal B_{3,3}+t^2\mathcal B_{3,2}+t\mathcal B_{3,1}.
\end{align*}
\end{example}
Formula (\ref{PA_n}) has reminiscence of a well known result on quasideterminants.
\begin{proposition}[Gelfand \emph{et al.} \cite{GGRW}]
\begin{equation}
\left|\begin{array}{ccccc}
a_{11}&a_{12}&a_{13}&\cdots&\boxed{a_{1n}}\\
-1&a_{22}&a_{23}&\cdots&a_{2n}\\
0&-1&a_{33}&\cdots&a_{3n}\\
&\cdots\\
0&\cdots&0&-1&a_{nn}
\end{array}\right|=a_{1n}+\sum_{1\leq j_1<\cdots<j_{k}<n}a_{1j_1}a_{j_1+1,2}\cdots a_{j_k+1,n}
\end{equation}
\end{proposition}
Furthermore, formula (\ref{B2P}) is an analogue of the result of Ebrahimi \emph{et al.}.
\begin{theorem}[Ebrahimi \emph{et al.} \cite{ebrahimi2014noncommutative}]
\[
\mathrm {MB}_{n}(1)=\left|\begin{array}{ccccc}
\binom{n-1}0d_1&\binom{n-1}1d_2&\binom{n-1}2d_3&\cdots&\boxed{\binom{n-1}{n-1}d_n}\\
-1&\binom{n-2}0d_1&\binom{n-2}1d_2&\cdots&\binom{n-2}{n-2}d_{n-1}\\
0&-1&\binom{n-3}0d_1&\cdots&\binom{n-3}{n-3}d_{n-3}\\
&\cdots\\
0&\cdots&0&-1&\binom{0}0d_1
\end{array}\right|.
\]
\end{theorem}
The connection between all these results remains to be investigated.

\subsection*{Acknowledgments}
This paper is partially supported by the PHC MAGHREB project IThèM and the ANR project CARMA.

\appendix
\section{Bell polynomials and coproducts in $Sym$}\label{Symmetric}

In fact, most of the identities on Bell polynomials can be obtained by manipulating generating functions and are  closely related to some other identities occurring in literature. Typically, the relation between the complete Bell polynomials $A_n(a_1,a_2,\dots)$ and the variables $a_1, a_2, \dots$ is very closely related to the Newton Formula which links the generating functions of complete symmetric functions $h_n$ (Cauchy series) to those of the power sums $p_n$. The symmetric functions form a commutative algebra $Sym$ freely generated by the complete functions $h_n$ or the power sum functions $p_n$. So, specializing the variable $a_n$ to some numbers is equivalent to specializing the power sum functions $p_n$.
More soundly, the algebra $Sym$ can be endowed with coproducts conferring to it a structure of Hopf algebra.
 For instance, the coproduct for which the power sums are primitive turns $Sym$ into a self-dual Hopf algebra.
  The coproduct can be translated in terms of generating functions by a product of two Cauchy series. 
  This kind of manipulations appears also in the context of Bell polynomials,
   for instance when computing the complete Bell polynomials of the sum of two sequences of
    variables $a_1+b_1$, $a_2+b_2,\dots$. Another coproduct turns $Sym$ into a non-cocommutative
     Hopf algebra called the Fa\`a di Bruno algebra which is related to the Lagrange inversion.
      Finally, the coproduct such that the power sums are group-like can be related also to a few other formulae on Bell polynomials. 
      The aim of this section is to investigate these connexions and in particular to restate some known results
       in terms of symmetric functions and virtual alphabets.
        We also give a few new results that are difficult to prove without the help of symmetric functions.
\subsection{Bell polynomials as symmetric functions}

First, let us recall some operations on alphabets.
Given two alphabets $\X$ and $\Y$, we also define (see \emph{e.g.} \cite{Lascoux}) the alphabet $\X + \Y$ by:
\begin{equation}
p_n(\mathbb{X}+\mathbb{Y}) = p_n(\mathbb{X})+ p_n(\mathbb{Y})
\end{equation}
and the alphabet $\alpha\X$ (resp $\X\Y$), for $\alpha \in \C$ by:
\begin{equation}\label{Somme}
p_n(\alpha\mathbb{X})=\alpha p_n(\mathbb{X}) (\mbox{ resp } p_n(\mathbb{X}\mathbb{Y})=p_n(\mathbb{X})p_n(\mathbb{Y})).
\end{equation}

In terms of Cauchy functions, these transforms imply
\begin{equation}
\sigma_t(\mathbb{X}+\mathbb{Y}) = \sigma_t(\mathbb{X}) \sigma_t(\mathbb{Y})
\end{equation}
and 
\begin{equation}\label{CauchyKernel}
\sigma_t(\mathbb{X}\mathbb{Y}) = \sum_\lambda \frac{1}{z_\lambda}p^\lambda(\mathbb{X})p^\lambda(\mathbb{Y}) t^{|\lambda|}.
\end{equation}

In fact $\sigma_t(\X\Y)$ encodes the kernel of the scalar product defined by 
$\langle p^\lambda , c_\mu \rangle = \delta_{\lambda, \mu}$ with $c_\lambda= \frac{p^\lambda}{z_\lambda}$.
Notice that $c_n = \frac{p_n}{n}$ and 
\begin{equation}\label{SymToC}
Sym= \mathbb{C}[c_1, c_2, \dots ].
\end{equation}

From (\ref{Bell}) and (\ref{newton}), we obtain
\begin{proposition}
$ h_n = \frac{1}{n!} A_n(1!c_1, 2!c_2, \dots)$.
 \end{proposition}
 
Conversely, Equality (\ref{SymToC}) implies that  the morphism  $\phi_a$ sending each $c_i$ to $\frac{a_i}{i!}$ is well defined for any sequence of numbers $a= (a_i)_{i \in \N \setminus\{0\}}$ and $\phi_a(h_n) = \frac{1}{n!} A_n(a_1, a_2, \dots)$. 
Let us define also $h_n^{(k)}(\X) = [\alpha^k] h_n(\alpha \X)$. From (\ref{newton}) and (\ref{Somme}) we have 
\[ h_n^{(k)} = \sum_{\lambda = [\lambda_1, \dots, \lambda_k] \vdash n }c_\lambda = [t^n] \frac{1}{k!} \left(\sum_{i \geq 1} c_i t^i \right)^k \]
and so, everything works as if we use a special (virtual) alphabet $\X^{(a)}$ satisfying $c_n(\X^{(a)}) = n! a_n$.
More precisely:
\begin{proposition}
\begin{equation}
 \phi_a (h_n^{(k)})= h_n^{(k)}(\X^{(a)}) =  \frac{1}{n!} B_{n,k} (a_1, \dots, a_k, \dots).
\end{equation}
\end{proposition}

\begin{example}\rm
Let $\mathbf 1$ be the virtual alphabet defined by $c_n(\mathbf 1)=\frac 1n$ for each $n\in \N$. In this case  the Newton Formula gives $h_n(\mathbf 1)=1$.
Hence $A_n(0!,1!,2!,\dots,(m-1)!,\dots)=n!$ and  $B_{n,k}(0!,1!,2!,\dots,(m-1)!,\dots)=n![\alpha^k][t^n]\left(1\over 1-t\right)^\alpha=s_{n,k}$
, the Stirling number of the first kind.  
\end{example}
\begin{example}\rm
A more complicated example is treated in \cite{bouroubi2006new,KC} where $a_i=i^{i-1}$. 
In this case, the specialization gives
$
\sigma_t(\alpha\X^{(a)})=\exp\{-\alpha W(-t)\}
$
where $W(t)=\sum_{n=1}^{\infty}(-n)^{n-1}{t^n\over n!}$ is the Lambert $W$ 
function satisfying $W(t)\exp\{W(t)\}=t$ (see \emph{e.g.} \cite{corlessetal1996Lambert}). Hence,
$
\sigma_t(\alpha\X^{(a)})=\left(W(-t)\over -t\right)^{\alpha}.
$
But the expansion of the series $\left(W(t)\over t\right)^\alpha$ is known to be:
\begin{equation}\label{W1}
\left(W(t)\over t\right)^\alpha=1+\sum_{n=1}^\infty\frac1{n!}\alpha(\alpha+n)^{n-1}(-t)^n.
\end{equation}
Hence, we obtain 
$B_{n,k}(1,2,3^2,\dots,m^{m-1},\dots)=\binom{n-1}{k-1}n^{n-k}.$
Note that the expansion of $W(t)$ and (\ref{W1}) are usually obtained by the use of the Lagrange inversion. 
\end{example}
\begin{example}\rm
With these notations we have 
$ B_{n,k}(a_1+b_1, \dots) = \frac{1}{n!} h_n^{(k)}(\X^{(a)} + \X^{(b)}),$
and classical properties of Bell polynomials can be deduced from symmetric functions through this formalism.
For instance, the equalities 
$ c_n(\X^{(a)} + \X^{(b)}) =   c_n(\X^{(a)}) +  c_n(\X^{(b)}) $
and
$h_n(\X^{(a)} + \X^{(b)}) = \sum_{i+j=n}  h_i(\X^{(a)}) h_j(\X^{(b)})$
give
\[A_n(a_1+ b_1 , \dots) = \sum_{i+j = n } \binom ni A_i(a_1, a_2, \dots) A_j(b_1,b_2, \dots) \]
and
\[ B_{n,k}(a_1 + b_1, \dots) = \sum_{r+s = k} \sum_{i+j = n}  \binom ni B_{i,r}(a_1, a_2, \dots) B_{j,s}(b_1,b_2, \dots).\]
\end{example}

\begin{example}\rm
 Another example is given by

 \[ \begin{array}{r}A_n(1 a_1 b_1 , 2 a_2 b_2 , \dots , m a_m b_m, \dots) 
 = n! \displaystyle\sum_{\lambda \vdash n}  \det\left| \frac{A_{\lambda_{i} - i +j } (a_1,a_2, \dots)}{ (\lambda_{i} - i +j )! } \right|   
 \\\times\displaystyle \det\left| \frac{A_{\lambda_{i} - i +j } (b_1,b_2, \dots)}{ (\lambda_{i} - i +j )! } \right|,\end{array}\]
using the convention $A_{-n}=0$ for $n>0$.
This formula is an emanation of the Jacobi-Trudi formula and is derived from the Cauchy kernel (\ref{CauchyKernel}), remarking that $c_n(\X^{(a)} \X^{(b)}) = n c_n(\X^{(a)}) c_n(\X^{(b)})$ and 
$$h_n(\X^{(a)} \X^{(b)}) =  \sum_{\lambda \vdash n} s_\lambda(\X^{(a)}) s_\lambda(\X^{(b)})
=  \sum_{\lambda \vdash n} \det  \left| h_{\lambda_{i} - i +j }^{(a)} (\X^{a}) \right| \det \left| h_{\lambda_{i} - i +j }^{(b)} (\X^{b})  \right|,$$
where $s_\lambda=\det  \left| h_{\lambda_{i} - i +j }^{(a)}  \right|$ is a Schur function (see \emph{e.g.}\cite{macdonald1998symmetric}).
\end{example}
\subsection{Other interpretations}
First we focus on the identity (\ref{partial Bell}) and we interpret it as the Cauchy function $\sigma_t(k\hat\X^{(a)})$ where $\hat\X^{(a)}$  is
the virtual alphabet such that $h_{i-1}(\hat\X^{(a)}) =\frac{a_i}{i!}$. 
This means that we consider the morphism $\hat\phi_a : Sym \longrightarrow \mathbb{C}$ sending $h_i$ to $\frac{a_{i+1}}{(i+1)!}$.
We suppose that $a_1 = 1$ otherwise we use (\ref{Cas1}) and (\ref{Cas2}). With these notations we have 

\begin{proposition}
\begin{equation}\label{B2hk}
 B_{n,k}(a_1, a_2, \dots) = \frac{n!}{k!} h_{n-k}(k\hat\X^{(a)}).
\end{equation}
\end{proposition}

\begin{example}
\rm
If $a_i=i$  we have $h_i(\hat\X^{(a)})=\frac1{i!}$ and so $\sigma_t(k\hat\X^{(a)})=\exp(kt)$. Hence, we recover the classical result
$$ B_{n,k}(1,2,\dots,m,\dots)=\binom nkk^{n-k}.$$ 
\end{example}
From $h_n(\X + \Y) = \sum_{i+j = n} h_i(\X) h_j(\Y)$ we deduce two classical identities:
\begin{equation}\label{BellareBinomial}
\binom {k_1+k_2}{k_1}B_{n,k_1+k_2}(a_1, a_2, \dots)=\sum_{i=0}^n\binom niB_{i,k_1}(a_1, a_2, \dots) B_{n-i,k_2}(a_1, a_2, \dots)
\end{equation} 
and 
\begin{equation}\label{convolutionBell}\begin{array}{l}
\displaystyle \binom nkB_{n-k,k}\left(a_1b_1,\dots,\frac1{m+1}\sum_{i=1}^m\binom {m+1}i a_ib_{m+1-i},\dots \right)
=\\\displaystyle\sum_{i=k}^{n-k}\binom niB_{i,k}(a_1,a_2,\dots)B_{i,k}(b_1,b_2,\dots).
\end{array}\end{equation}
Indeed, formula (\ref{BellareBinomial}) is obtained by setting $\X= k_1 \hat\X^{(a)}$ and $\X= k_2 \hat\X^{(a)}$. 
Formula (\ref{convolutionBell}) is called the convolution formula for Bell polynomials (see \emph{e.g.} \cite{ThesisMih}) and is obtained by setting
$\X= \hat\X^{(a)}$ and $\Y= \hat\X^{(b)}$ in the left hand side and $\X= k \hat\X^{(a)}$ and $\Y= k \hat\X^{(b)}$ in the right hand side. \begin{example}\rm
 The partial Bell polynomials are known to be involved in interesting identities on binomial functions.
 Let us first recall that a  binomial sequence is a family of functions $(f_{n})_{n\in\N}$ satisfying $f_{0}(x)= 1 $ and  
\begin{equation}\label{binomfunc} 
f_n(a+b)=\sum_{k=0}^n\binom nkf_k(a)f_{n-k}(b),
\end{equation}
for all $a,b \in \mathbb{C}$ and $n \in \mathbb{N}$.
Setting $h_n(\mathbb A):=\frac{f_n(a)}{n!}$ and $h_n(\mathbb B):=\frac{f_n(b)}{n!}$, with these notations $ f_n(ka) = n! h_n(k \mathbb A)$.
Hence,
\begin{equation}\label{Mih23}
B_{n,k}(1,\dots,if_{i-1}(a),\dots) = \frac {n!}{k!} h_{n-k}(k\mathbb A)=\binom{n}{k} f_{n-k}(ka).
\end{equation}
Notice that from (\ref{BellareBinomial}), $f_n(k)= \left\{\begin{array}{ll}
{\binom nk}^{-1}B_{n,k}(a_1,a_2,\dots) & \mbox{ if } n>0 \\
1 & \mbox{ if } n=0 
\end{array}\right. $   , is binomial and we obtain
\begin{equation}\label{BofB}\begin{array}{l}
{\binom n{k_1k_2}}^{-1}B_{n,k_1}(1,\dots,i{\binom {i-1}k}^{-1}B_{i-1,k_2}(a_1,a_2,\dots),\dots)=\\
{\binom {n-k_1}{k_1k_2}}^{-1}B_{n-k_1,k_1k_2}(a_1,a_2,\dots). \end{array}
\end{equation}
Several related identities are compiled in \cite{ThesisMih}.
\end{example}

\begin{example}\rm
Taking the coefficient of $t^{n-k-1}$ in the left hand side and the right hand side of the equality
$
{d\over dt}\sigma_t((k+1)\mathbb X)=(k+1)\left({d\over dt}\sigma_t(\X)\right)\sigma_t(k\X),
$
we obtain
$$
(n-k)h_{n-k}((k+1)\X)=(k+1)\sum_{i=1}^{n-k}ih_i(\X)h_{n-i-k}(\X)
$$
 and we recover the identity (see \emph{e.g.} \cite{cvijovic2011new}):\begin{equation}\label{EqCv}
B_{n,k}(a_1,a_2,\dots)=\frac1{n-k}\sum_{i=1}^{n-k}\binom ni\left[(k+1)-{n+1\over i+1}\right](i+1)a_iB_{n-i,k}(a_1,a_2,\dots).
\end{equation}

\end{example}

\begin{example}\rm
Let $(a_n)_{n>0}$ and $(b_n)_{n>0}$ be two sequences of numbers such that $a_1=b_1=1$ and $d_n=n!\sum_{\lambda\vdash n-1}\det\left|a_{\lambda_i-i+j+1}\over (\lambda_i+j+1)!\right|
\det\left|b_{\lambda_i-i+j+1}\over (\lambda_i+j+1)!\right|$ with the convention $a_{-n}=b_{-n}=0$ if $n\geq0$. The Cauchy kernel and the orthogonality of Schur functions give
\[\begin{array}{r}
B_{n,k}(d_1,d_2,\dots)={n!\over k!}\displaystyle\sum_{\lambda\vdash n-k}(k_1!k_2!)^{\ell(\lambda)}\det\left|B_{\lambda_i-i+j+k_1,k_1}(a_1,a_2,\dots)\over(\lambda_i-i+j+k_1)!\right|\\\displaystyle
\det\left|B_{\lambda_i-i+j+k_1,k_1}(b_1,b_2,\dots)\over(\lambda_i-i+j+k_1)!\right|,\end{array}
\]
for any $k_1k_2=k$. Indeed, it suffices to use the fact that $$h_n(k\X^{(a)}\X^{(b)})=\sum_{\lambda\vdash n}s_\lambda(k_1\X^{(a)})s_\lambda(k_2\X^{(b)}).$$
\end{example}

The sum $\X+\Y$ and the product $\X\Y$ of alphabets are two examples of coproducts endowing $Sym$ with a structure of Hopf algebra. The sum of alphabets encodes the coproduct $\Delta$ for which the power sums are of type Lie (\emph{i.e.} $\Delta(p_n)=p_n\otimes 1+1\otimes p_n\sim p_n(\X+\Y)=p_n(\X)+p_n(\Y)$ by identifying $f \otimes g $ with $f(\X) g(\Y)$) whilst the product of alphabets encodes the coproduct $\Delta'$ for which the power sums are group-like (\emph{i.e.} $\Delta'(p_n)=p_n\otimes  p_n\sim p_n(\X\Y)=p_n(\X)p_n(\Y)$).

The algebra of symmetric functions can be endowed with another coproduct that confers a structure
of Hopf algebra: this is the Faà di Bruno algebra \cite{doubilet1974hopf, joni1979coalgebras}. This algebra is rather important since it is related to
the Lagrange-Bürmann formula. The Bell polynomials also appear in this context. As a consequence,
one can define a new operation on alphabets corresponding to the composition of Cauchy generating
functions. Let $\X$ and $\Y$ be two alphabets and set $f(t)=t\sigma_t(\X)$ and $g(t)=t\sigma_t(\Y)$. The composition $\X\circ \Y$ is defined by $\sigma_t(\X\circ\Y)=\frac1tf\circ g(t).$
The relationship with Bell polynomials can be established by observing that we have 
$$\frac1tf\circ g=\sum_{n\geq 0}\left(\sum_{k=1}^{n+1}{k!\over (n+1)!}h_{k-1}(\X)B_{n+1,k}(1,2h_1(\Y),3!h_2(\Y),\dots)\right)t^n.$$
Equivalently, $h_n(\X\circ \Y)=\sum_{k=0}^n{(k+1)!\over (n+1)!}h_k(\X)B_{n+1;k+1}(1,2!h_1(\Y),3!h_2(\Y),\dots)$.

The antipode of the Fa\`a di Bruno algebra is also described in terms of alphabets as the operation which associates to each alphabet $\X$ the alphabet $\X^{\langle-1\rangle}$ satisfying $\sigma_t(\X\circ\X^{\langle -1\rangle})=1$. More explicitly, one has
\begin{equation}\label{antipodeFaa}h_n(\X^{\langle -1\rangle})={n!\over (2n+1)!(n+1)}B_{2n+1,n}(1,-2!e_1(\X),3!e_2(\X),\dots),\end{equation}
 where $e_n(\X)$ is the elementary symmetric function defined by $\sum_ne_n(\X)t^n=\frac1{\sigma_{-t}(\X)}$.
\begin{example}\rm
Let $\omega(t)=t\sigma_t(\X)$. The Lagrange inversion consists in finding an alphabet $\X'$ such that $\phi(t)=\sigma_t(\X')$. According to (\ref{antipodeFaa}), it suffices to set $\X'=-\X^{\langle -1\rangle}$.\\
Let $F(t)=\sigma_t(\Y)$. When stated in terms of alphabets, the Lagrange-B\"urmann formula reads
\[
F(\omega(t))=1+\sum_{n\geq 1}{d^{n-1}\over du^{n-1}}[\sigma'_u(\Y)\sigma_u(-n\X^{\langle-1\rangle})]|_{u=0}{t^n\over n!}.
\]
In other words one has,
$
h_{n-k}(-n\X^{\langle -1\rangle})={(k-1)!\over(n-1)!}B_{n,k}(1,2!h_1(\X),3!h_3(\X),\dots).
$
So we recover a result due to Sadek Bouroubi and Moncef Abbas \cite{bouroubi2006new}: 
\[B_{n,k}(1,h_1(2\X),\dots,(m-1)!h_{m-1}(m\X),\dots)={(n-1)!\over (k-1)!}h_{n-k}(n\X).\]
\end{example}


\begin{thebibliography}{abc}

\bibitem {bell1934exponential}
Eric~Temple Bell.
\newblock Exponential polynomials.
\newblock {\em Annals of Mathematics}, pages 258--277, 1934.

\bibitem {bergeron2008invariants}
Nantel Bergeron, Christophe Reutenauer, Mercedes Rosas, and Mike Zabrocki.
\newblock Invariants and coinvariants of the symmetric groups in noncommuting
  variables.
\newblock {\em Canad. J. Math.}, 60(2):266--296, 2008.
\bibitem{BLL} Fran\c cois Bergeron and Gilbert Labelle and Pierre Leroux, {\it 
Introduction to the Theory of
Species of Structures}, UQ\`AM Univerit\'e du Qu\'ebec \`a Montr\'eal, 2013.
\bibitem {bouroubi2006new}
Sadek Bouroubi and Moncef Abbas.
\newblock New identities for {B}ell’s polynomials. {N}ew approaches.
\newblock In {\em Rostock. Math. Kolloq}, volume~61, pages 49--55, 2006.
\bibitem {bultel2013redfield}
Jean-Paul Bultel, Ali Chouria, Jean-Gabriel Luque, and Olivier Mallet.
\newblock Word symmetric functions and the {R}edfield-{P}{\'o}lya theorem.
\newblock {\em FPSAC'2013, DMTCS Proceedings}, (01):563--574, 2013.

\bibitem {corlessetal1996Lambert}
R.~M. Corless, G.~H. Gonnet, D.~E. Hare, D.~J. Jeffrey, and D.~E. Knuth.
\newblock On the {L}ambert ${W}$ function.
\newblock {\em Advances in Computational Mathematics}, 5:329--359, 1996.

\bibitem {comtet1974advanced}
Louis Comtet.
\newblock {\em Advanced Combinatorics: {T}he art of finite and infinite
  expansions}.
\newblock Springer, 1974.

\bibitem {cvijovic2011new}
Djurdje Cvijovi{\'c}.
\newblock New identities for the partial {B}ell polynomials.
\newblock {\em Applied Mathematics Letters}, 24(9):1544--1547, 2011.

\bibitem {doubilet1974hopf}
Peter Doubilet.
\newblock A {H}opf algebra arising from the lattice of partitions of a set.
\newblock {\em Journal of Algebra}, 28(1):127--132, 1974.

\bibitem {ebrahimi2014noncommutative}
Kurusch Ebrahimi-Fard, Alexander Lundervold, and Dominique Manchon.
\newblock Noncommutative {B}ell polynomials, quasideterminants and incidence
  {H}opf algebras.
\newblock {\em International Journal of Algebra and Computation},
  24(05):671--705, 2014.

\bibitem {faadibruno1855sviluppo}
Francesco Fa\`a~di Bruno.
\newblock Sullo sviluppo delle funzioni.
\newblock {\em Annali di Scienze Matematiche e Fisiche}, 6:479--480, 1855.

\bibitem {faadibruno1857note}
Francesco Fa\`a~di Bruno.
\newblock Note sur une nouvelle formule de calcul différentiel.
\newblock {\em Quarterly J. Pure Appl.Math.}, 1:359--360, 1857.

\bibitem{Foissy2007} L. Foissy, {\em Bidendriform bialgebras, trees, and free quasi-symmetric functions}, J.Pure Appl. Algebra 209 no. 2: 439-459, 2007.

\bibitem{GGRW} Israel Gelfand, Sergei Gelfand, Vladimir Retakh, and R Wilson, {\it Quasideterminants}, Advances in Mathematics, 193(1):56–141, 2005.

\bibitem {grossman1989hopf}
Robert Grossman and Richard~G Larson.
\newblock Hopf-algebraic structure of families of trees.
\newblock {\em Journal of Algebra}, 126(1):184--210, 1989.

\bibitem {hivert2008commutative}
Florent Hivert, Jean-Christophe Novelli, and Jean-Yves Thibon.
\newblock Commutative combinatorial {H}opf algebras.
\newblock {\em Journal of Algebraic Combinatorics}, 28(1):65--95, 2008.

\bibitem {harris1967number}
Bernhard Harris and Lowell Schoenfeld.
\newblock The number of idempotent elements in symmetric semigroups.
\newblock {\em Journal of Combinatorial Theory}, 3(2):122--135, 1967.

\bibitem {joni1979coalgebras}
SA~Joni and Gian-Carlo Rota.
\newblock Coalgebras and bialgebras in combinatorics.
\newblock {\em Stud. Appl. Math}, 61(2):93--139, 1979.

\bibitem {KC}
S.~Khelifa and Y.~Cherruault.
\newblock New results for the {A}domian method.
\newblock {\em Kybernetes.}, 29(n3):332--355, 2000.

\bibitem {Lascoux}
A.~Lascoux.
\newblock {\em Symmetric Functions and Combinatorial Operators on Polynomials}.
\newblock CBMS Regional Conference Series in Mathematics, American Mathematical
  Society, 2003.

\bibitem{Loday}  J.-L. Loday, {\em Cup-product for Leibniz cohomology and dual Leibniz algebras }, Math. Scand., vol. 77, no 2: 189-196, 1995.

\bibitem {macdonald1998symmetric}
Ian~Grant Macdonald.
\newblock {\em Symmetric functions and {H}all polynomials}.
\newblock Oxford {U}niversity {P}ress, 1998.

\bibitem {mihoubi2008bell}
Miloud Mihoubi.
\newblock Bell polynomials and binomial type sequences.
\newblock {\em Discrete Mathematics}, 308(12):2450--2459, 2008.

\bibitem {ThesisMih}
Miloud Mihoubi.
\newblock Polyn{\^o}mes multivari{\'e}s de {B}ell et polyn{\^o}mes de type
  binomial.
\newblock {\em Th{\`e}se de Doctorat, L'Universit{\'e} des Sciences et de la
  Technologie Houari Boumediene}, 2008.

\bibitem {mihoubi2010partial}
Miloud Mihoubi.
\newblock Partial {B}ell polynomials and inverse relations.
\newblock {\em J. Integer Seq}, 13(4), 2010.

\bibitem{NovelliThibon2006} J.-C. Novelli,
J.-Y. Thibon
,
{\em Construction de trig\`ebres dendriformes}
, C. R. Acad. Sci,
Paris, S\'er. I,
342
, (2006), 365–369.

\bibitem{NovelliThibon2007}  J.-C. Novelli,
J.-Y. Thibon
, {\em Hopf algebras and dendriform structures arising from parking functions}, Fundamenta Mathematic\ae 193: 189-241, 2007 

\bibitem {riordan1946derivative}
John Riordan.
\newblock Derivatives of composite functions.
\newblock {\em Bulletin of American Mathematical Society}, 52:664--667, 1946.

\bibitem {riordan1958introduction}
John Riordan.
\newblock {\em An Introduction to Combinatorial Analysis}.
\newblock John Wiley \& Sons, New York, 1958 ; Princeton University Press,
  Princeton, NJ, 1980.


\bibitem {Sloane}
N.~J.~A. Sloane.
\newblock The on-line encyclopedia of integer sequences, 2003.

\bibitem {tainiter1968generating}
Melvin Tainiter.
\newblock Generating functions on idempotent semigroups with application to
  combinatorial analysis.
\newblock {\em Journal of Combinatorial Theory}, 5(3):273--288, 1968.

\bibitem {wolf1936symmetric}
Margarete~C Wolf.
\newblock Symmetric functions of non-commutative elements.
\newblock {\em Duke Mathematical Journal}, 2(4):626--637, 1936.

\bibitem{MK} H. Munthe-Kaas, {\it  Lie-Butcher theory for Runge-Kutta methods},
BIT Numerical Mathematics, 35(4):572–587, 1995.

\bibitem {wang2009general}
Weiping Wang and Tianming Wang.
\newblock General identities on {B}ell polynomials.
\newblock {\em Computers and Mathematics with Applications}, 58(1):104--118,
  2009.
\end{thebibliography}
\end{document}